\documentclass[journal,twoside,web]{ieeecolor}

  \makeatletter
  \let\NAT@parse\undefined
  \makeatother

\usepackage{generic}
\usepackage{cite}
\usepackage{amsmath,amssymb,amsfonts}
\usepackage{algorithmic}
\usepackage{graphicx}
\usepackage{textcomp}

\usepackage[english]{babel}
\usepackage[T1]{fontenc}

\usepackage{amsthm}
  \theoremstyle{plain}
  \newtheorem{theorem}{Theorem}
  \newtheorem{lemma}{Lemma}
  \newtheorem{proposition}[theorem]{Proposition}
  \newtheorem{corollary}{Corollary}
  \newtheorem{assumption}{Assumption}
  \newtheorem{definition}{Definition}
  \newtheorem{example}{Example}
  \newtheorem{convention}{Convention}
  \theoremstyle{remark}
  \newtheorem{remark}{Remark}[section]
\usepackage{mathrsfs}
\usepackage{mathtools}

\usepackage{etoolbox}
\makeatletter
\@ifundefined{color@begingroup}%
  {\let\color@begingroup\relax
   \let\color@endgroup\relax}{}%
\def\fix@ieeecolor@hbox#1{%
  \hbox{\color@begingroup#1\color@endgroup}}
\patchcmd\@makecaption{\hbox}{\fix@ieeecolor@hbox}{}{\FAILED}
\patchcmd\@makecaption{\hbox}{\fix@ieeecolor@hbox}{}{\FAILED}
\usepackage{algorithm}
  
\usepackage{booktabs}

\begin{document}
\bstctlcite{IEEEexample:BSTcontrol}

\title{Probabilistic Framework of Howard's Policy Iteration: BML Evaluation and Robust Convergence Analysis}
\author{Yutian Wang, Yuan-Hua Ni, Zengqiang Chen and Ji-Feng Zhang, \IEEEmembership{Fellow, IEEE}
\thanks{This work was supported in part by National Key R\&D Program of China under Grant 2018YFA0703800, and in part by the National Natural Science foundation of China under Grants 62173191 and 61973175.}
\thanks{Yutian Wang, Yuan-Hua Ni and Zengqiang Chen are with the College of Artificial Intelligence, Nankai University, Tianjin, China (e-mail: wangyt2239@mail.nankai.edu.cn; yhni@nankai.edu.cn; chenzq@nankai.edu.cn).}
\thanks{Ji-Feng Zhang is with the Key Laboratory of Systems and Control, Academy of Mathematics and Systems Science, Chinese Academy of Sciences, Beijing 100190, China, and also with the School of Mathematical Sciences, University of Chinese Academy of Sciences, Beijing
100149, China (e-mail: jif@iss.ac.cn).}
}
\maketitle
\begin{abstract}
This paper aims to build a probabilistic framework for Howard's policy iteration algorithm using the language of forward-backward stochastic differential equations (FBSDEs). As opposed to conventional formulations based on partial differential equations, our FBSDE-based formulation can be easily implemented by optimizing criteria over sample data, and is therefore less sensitive to the state dimension. In particular, both on-policy and off-policy evaluation methods are discussed by constructing different FBSDEs. The backward-measurability-loss (BML) criterion is then proposed for solving these equations. By choosing specific weight functions in the proposed criterion, we can recover the popular Deep BSDE method or the martingale approach for BSDEs. The convergence results are established under both ideal and practical conditions, depending on whether the optimization criteria are decreased to zero. In the ideal case, we prove that the policy sequences produced by proposed FBSDE-based algorithms and the standard policy iteration have the same performance, and thus have the same convergence rate. In the practical case, the proposed algorithm is  still proved to converge robustly under mild assumptions on optimization  errors.
\end{abstract}


\section{Introduction}\label{sec:introduction}
As an abstract description of policy-based methods, such as policy iteration (PI) \cite{Howard1960,Puterman1979,santos_rust_2004,lee_sutton_2021,Sutton2018,wei_liu_2015} and policy gradient methods \cite{sutton_mcallester_1999,Silver2014,Kakade2001}, the general policy iteration (GPI) for optimal control problems works as follows.
\begin{enumerate}
\item \textit{(Initialization.)} Given an initial policy $\alpha^0$ and set $n\leftarrow0$.
\item \textit{(GPI Subroutine.)} Given a policy $\alpha^{n-1}$, find a new policy $\alpha^n$ in the policy space $\mathcal{A}$.
\item Set $n\leftarrow n+1$ and go back to step~2.
\end{enumerate}
The key element of GPI is step~2, referred to as the GPI subroutine in this paper, which takes the current policy $\alpha^{n-1}$ as inputs, along with some other arguments if needed, and returns a new policy $\alpha^{n}$. For example, in policy gradient methods, the new policy is obtained via gradient descent in the policy space. That subroutine is carefully designed such that the generated policy sequence $\{\alpha^n\}$ of GPI converges to, or approaches in some sense, an optimal policy $\alpha^*$.

Originally developed by Howard for the Markov process model \cite{Howard1960}, Howard's policy improvement procedure (an instance of GPI subroutines), along with the policy iteration method, has been widely applied to optimal control problems, from disctete to continous, deterministic to stochastic and linear to nonlinear systems \cite{Kleinman1968,lewis_vrabie_2009,ni_fang_2013,li_peng_2019a,li_peng_2019b}. A major advantage of Howard's policy iteration (hereafter referred to as the standard PI) is its fast convergence rate. For discrete time and state problems, Puterman and Brumelle\cite{Puterman1979} pointed out that the standard PI can be regarded as an instance of Newton's method, noting that both are finding zeros of a nonlinear operator. Based on this crucial observation, they successfully established a local quadratic convergence rate, which is also a standard result for Newton's iterative scheme in root finding problems. For linear quadratic regulation (LQR) problems in continuous time and state, the value function squence generated by the standard policy iteration also converges quadratically \cite{Kleinman1968}. Another interesting property of the standard PI is its robustness against numerical errors. For stochastic nonlinear systems,  Kerimkulov \textit{et al}.~\cite{Kerimkulov2020} analyzed the standard PI with perturbation errors. They employed the theory of backward stochastic differential equations (BSDEs) to estimate the performance error bound; see also \cite{Pang2022} for perturbation discussion on continous-time LQR problem.

Howard's policy improvement procedure is usually recognized as two consecutive steps: policy evaluation and policy improvement. The purpose of policy evaluation is to collect quantitative information on the current policy, or more specifically, the value function of the policy. Based on this information, the policy improvement step constructs a new policy that guarantees a monotone increase in performance. In this work, we focus on policy evaluation, and assume that a minimizing function for policy improvement exists and is accessible \cite{lee_sutton_2021,Kerimkulov2020}. Most early methods of policy evaluation obtain value functions by solving the differential Bellman equation, a first or second order linear partial differential equation (PDE) \cite{abu_lewis_2005,Kleinman1968,leake_liu_1967}. Since traditional finite difference methods for PDEs generally suffer from the curse of dimensionality \cite{E2021}, integral PI \cite{lewis_vrabie_2009} and temporal difference learning \cite{Doya2000,Fremaux2013} are preferred in practice. In addition to aforementioned works that focus on deterministic case, Jia and Zhou \cite{jia_zhou_2022} investigated policy evaluation in stochastic settings with a finite planning horizon. They extended temporal difference learning to stochastic systems, and proposed a martingale approach which can be viewed as the stochastic counterpart of integral PI. It is worth noting that their martingale approach utilized a forward-backward stochastic differential equation (FBSDE), which is precisely the stochastic representation of the value function. From this point of view, their work is closely related to early policy evaluation methods utilizing PDEs, as Feynman-Kac's formula relates FBSDEs and PDEs \cite{karatzas_shreve_1998}. On the other hand, Han~\textit{et al.}~\cite{Han2018} proposed Deep BSDE method as a numerical approach for high-dimensional PDEs, where the problem is transformed into an optimization problem subject to FBSDEs by nonlinear Feynman-Kac's formula. 

\textbf{Contributions.} The main contributions of this paper are as follows. \textbf{1)} Motivated by these two parallel applications of Feynman-Kac type formulae \cite{Han2018,jia_zhou_2022}, we rigorously build the FBSDE-based framework of policy evaluation. In particular, we propose two FBSDE-based GPI subroutines are proposed that, under certain assumptions, are shown to be equivalent to conventional PDE-based subroutine used in Howard's policy iteration. This in turn shows GPI equipped with proposed subroutines converges as fast as the standard PI. \textbf{2)} We propose a novel optimization-based formulation of policy evaluation, whereby value function gradients are evaluated rather than the value function itself. In the case of inexact policy evaluation, we present a robust convergence result in terms of the optimization errors. \textbf{3)} We propose a versatile criterion for the optimization problem in policy evaluation. As the solution to the FBSDE constraint is not known a priori, we prove that it is equivalent to optimizing the proposed backward-measurability-loss (BML) criterion. By selecting different weight functions in the BML criterion, we are able to recover the Deep BSDE method in \cite{Han2018} as well as the martingale approach in \cite{jia_zhou_2022}. Combining with the time discretization scheme in \cite{bender_zhang_2008}, our method can also be used to solving FBSDEs and Feynman-Kac type PDEs.
See also Figure~\ref{fig:0} for an overview of our policy iteration framework.

\textbf{Organizations.} This paper is organized as follows. In Section~\ref{sec:SOC}, we set up the stochastic optimal control problem and review the concept of value functions. In Section~\ref{sec:PDE-subroutine}, we state the standard policy iteration algorithm and present a global linear convergence result. Two FBSDE-based policy iteration algorithms are introduced and analyzed in Section~\ref{sec:FBSDE-subroutines}. In addition to the ideal convergence results, a robust convergence analysis is offered regarding optimization errors. Section~\ref{sec:approximation} dicusses the optimization problems in proposed algorithms. Numerical examples are present in Section~\ref{sec:simulations}. Finally, we conclude with some future directions in Section~\ref{sec:conclusions}.

\textbf{Notations.} Notations to be used frequently are summarized as follows. \textbf{1)} About probability theory and stochastic analysis. An element $\xi\in L^2_{\mathcal{F}}$ is a $\mathcal{F}$-measurable function with $ \operatorname{\mathbb{E}}\|\xi\|^2 < \infty$. $W^{t,T}\equiv\{W^{t,T}_s:t\leq s\leq T\}$ denotes a $d$-dimensional Brownian motion starting at $W_t^{t,T}=0$.  $\mathbb{S}^2(t,T)$ denotes the set of adapted process $Y$ satisfying $ \operatorname{\mathbb{E}}[\sup_{t\leq s\leq T}|Y_s|^2] < \infty$. $\mathbb{H}^2(t,T)$ denotes the set of adapted process $Z$ satisfying $\operatorname{\mathbb{E}}\int_t^T\|Z_s\|^2\,ds < \infty$. When there is no ambiguity, we drop the dependencies on $t$ and $T$ in these notations. \textbf{2)} About optimal control and reinforcement learning. We use $x\in\mathbb{R}^n$ and $a\in\mathbb{R}^m$ to denote the state and the action (control). A function $\alpha$ is termed a (feedback-control) policy if it maps time-state pairs to control values. We use $F^\alpha$ to indicate that a quantity $F$ depends on a policy $\alpha$ and $F^*$ to indicate the quantity corresponding to the optimal policy. Moreover, for a quantity $F(t,x,a)$ depending on the time-state-action triple, we write $F^a(\cdot,\cdot)\equiv F(\cdot,\cdot,a)$ and $F^\alpha(\cdot,\cdot)\equiv F(\cdot,\cdot,\alpha(\cdot,\cdot))$ if $a$ is a control value and $\alpha$ is a control policy. \textbf{3)} About vector space. For elements in Euclidean space, $\|\cdot\|$ stands for the $L^2$ norm and $\langle\cdot,\cdot\rangle$ stands for the standard inner product. \textbf{4)} About functional classes. We use $w\in C^{1,2}$ to say that $w$ is continuously differentiable with respect to the first variable and twice continuously differentiable with respect to the second varaible. In Section~\ref{sec:reformulation}, we also introduce the notation $\phi\in C^{\mathrm{UniLip}}_b$ to say that $\phi$ is uniformly Lipschitz continuous and uniformly bounded.

\begin{figure}
  \centering
  \includegraphics[width=.45\textwidth]{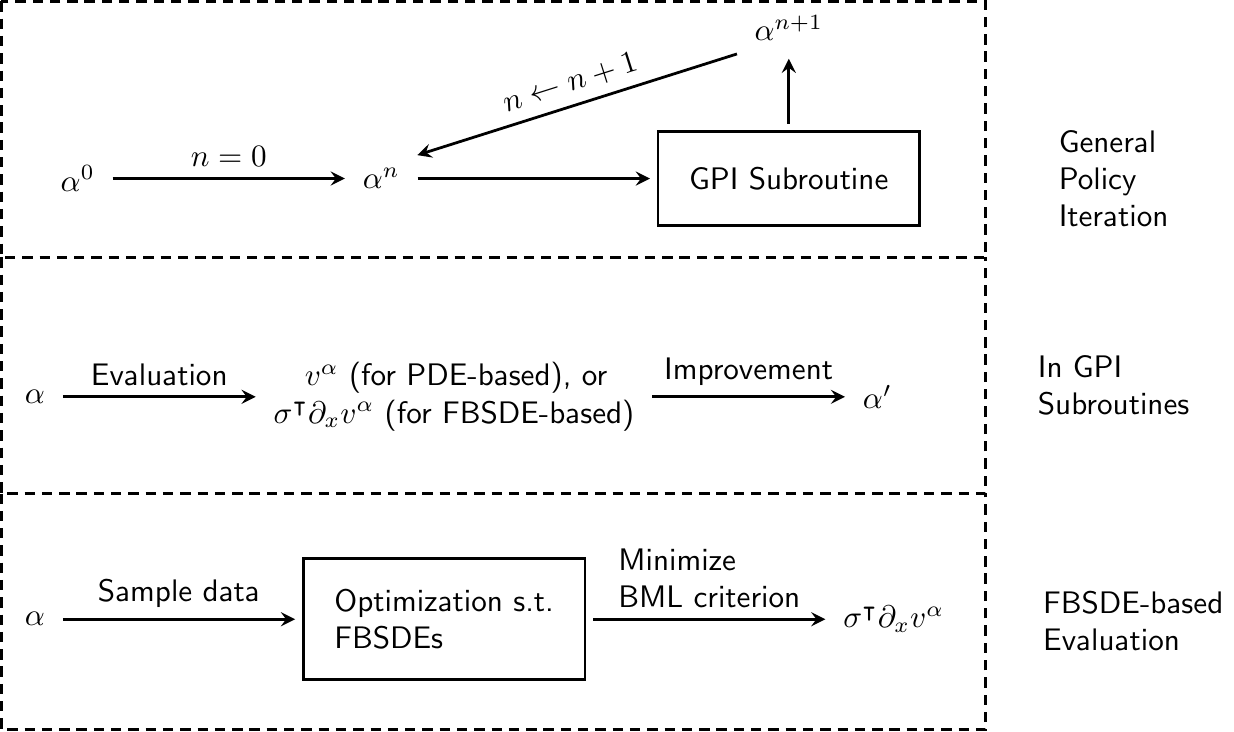}
  \caption{Hierarchical illustration of the proposed policy iteration framework. At the top level of the hierarchy is GPI, which iterates in the policy space. At the midlevel is the GPI subroutine, and at the bottom is the optimization formulation of policy evaluation.\label{fig:0}}
\end{figure}

\section{Preliminaries}\label{sec:SOC}
In this section, we review some basic concepts and results in general stochastic optimal control theory. For a comprehensive description of this subject, please refer to the monograph \cite{yong_zhou_1999}.

\subsection{Problem settings}

We consider a optimal control problem with system dynamics governed by the stochastic differential equation (SDE):
\begin{equation}
  \label{eq:general-SDE}
  X_s = x + \int_t^sb^\alpha(\tau,X_\tau)\,d\tau + \int_t^s\sigma(\tau,X_\tau)\,dW_\tau.
\end{equation}
The solution to this equation, denoted by $X^{\alpha,t,x}$ or simply $X^\alpha$, is a controlled diffusion process, depending on both the policy $\alpha$ and the starting point $(t,x)$. Let us fix the initial time-state pair $(t,x)$ at first. Eq.~\eqref{eq:general-SDE} is studied on an underlying probability space $(\Omega,\mathcal{F},\mathbb{P})$, which is required to be complete and admit a standard $d$-dimensional Brownian motion $\{W_s\}_{t\leq s\leq T}$ with $W_t=0$. Here, $T<\infty$ is the planning horizon. We equip $(\Omega,\mathcal{F},\mathbb{P})$ with the natural filtration $\{\mathcal{F}_s\}_{t\leq s\leq T}$ generated by $\{W_s\}_{t\leq s\leq T}$. Note that the definition of $\{W_s,\mathcal{F}_s;t\leq s\leq T\}$ relies on the choice of $t\in[0,T]$.\footnote{This is known as the weak formulation of stochastic optimal control problems in \cite{yong_zhou_1999}. The main motivation of this formulation is that we can deal with a family of stochastic optimal control problems by varying $(t,x)$.} We develop our theory with fixed $(t,x)$ and the generalization to varying $(t,x)$ is straightforward by substituting specific values.

The (controlled) drift coefficient $b^\alpha$ and diffusion coefficient $\sigma$ are measurable functions defined on $[0,T]\times\mathbb{R}^n$. In particular, $b^\alpha$ is defined by another measurable functions  $b:[0,T]\times\mathbb{R}^n\times\mathbb{R}^m\to\mathbb{R}^n$ and a policy $\alpha:[0,T]\times\mathbb{R}^n\to\mathbb{R}^m$, i.e.,  $b^\alpha:(t,x)\mapsto b(t,x,\alpha(t,x))$. Under certain conditions on $b^\alpha$ and $\sigma$, there exists an adapted process $X^{\alpha,t,x}$ satisfying Eq.~\eqref{eq:general-SDE} $\mathbb{P}$-a.s. for any $s\in[t,T]$; see, for example, Karatzas and Shreve \cite{karatzas_shreve_1998}. Here, by saying a process is adapted, we mean it is progressively measurable\footnote{Strictly speaking, an adapted process need not to be progressively measurable. But, if it is also measurable, then it has a stochastic equivalent process which is indeed progressively measurable \cite{Meyer1966}.}.

The cost of a policy $\alpha$ starting at $(t,x)$ is measured by the following expectation:
\begin{equation}
  \label{eq:def-valuefunction}
  v^\alpha(t,x) \coloneqq \operatorname{\mathbb{E}}\biggl[\int_t^Tf^\alpha(s,X_s^{\alpha,t,x})\,ds + g(X_T^{\alpha,t,x})\biggr].
\end{equation}
Here, $f:[0,T]\times\mathbb{R}^n\times\mathbb{R}^m\to\mathbb{R}$ and $g:\mathbb{R}^n\to\mathbb{R}$ are measurable functions, and $f^\alpha$ is defined in terms of $f$ and $\alpha$, in the same way as $b^\alpha$ is defined in terms of $b$ and $\alpha$. A control policy is said to be admissible if it takes value in $A\subset \mathbb{R}^m$ and the solution to Eq.~\eqref{eq:general-SDE} uniquely exists. We denote by $\mathcal{A}$ the collection of all admissible policies. When the policy $\alpha$ is fixed, the function $v^\alpha:[0,T]\times\mathbb{R}^n\to\mathbb{R}$ is called the value function of $\alpha$. In addition, the following infimum:
\begin{equation}
  \label{eq:def-optimal-valuefunction}
  v^*(t,x)\coloneqq \inf_{\alpha\in\mathcal{A}} v^\alpha(t,x)
\end{equation}
is called the optimal value function. 

The stochastic optimal control problem, in view of Eq.~\eqref{eq:general-SDE}--\eqref{eq:def-optimal-valuefunction}, is then stated as finding $\alpha^*\in\mathcal{A}$ such that $v^*(t,x) = v^{\alpha^*}(t,x)$ for a given pair $(t,x)$.

\subsection{Characterizing value functions via PDEs} 

Using dynamic programming, we can link value functions to a family of PDEs. Specifically, the dynamic programming principle states that
\begin{equation}
  \label{eq:DPP}
  \begin{aligned}
  v^*(t,x) = \inf_{\alpha\in\mathcal{A}}\operatorname{\mathbb{E}}\biggl[\int_t^{t+\epsilon}\kern-.5em f^\alpha(s,X_s^{\alpha,t,x})\,ds + v^*(t+\epsilon, X_{t+\epsilon}^{\alpha,t,x})\biggr].
  \end{aligned}
\end{equation}
Recall that for any sufficient smooth $v$, there is $\mathscr{L}^\alpha v(t,x)=\lim_{\epsilon\to0} \frac{1}{\epsilon} \operatorname{\mathbb{E}}\bigl[v(t+\epsilon, X_{t+\epsilon}^{\alpha,t,x}) - v(t,x)\bigr]$ with $\mathscr{L}^\alpha$ the infinitesimal generator the  associate to Eq.~\eqref{eq:general-SDE}
\begin{equation}
  \label{eq:def-L}
  \mathscr{L}^\alpha v\coloneqq \partial_tv + \langle b^\alpha,\partial_xv\rangle + \frac{1}{2} \operatorname{tr}\{\sigma\sigma^\intercal\partial_{xx}v\}.
\end{equation}
Here, we drop the dependency on $(t,x)$ for simplicity. Dividing Eq.~\eqref{eq:DPP} by $\epsilon$ and taking $\epsilon\to 0$ leads to a second order partial differential equation. Setting $t=T$ in the definition \eqref{eq:def-optimal-valuefunction} yields a boundary condition. Putting these all together and varying $(t,x)$ lead to the following second order nonlinear Cauchy problem for the optimal value function
\begin{equation}
  \label{eq:HJB}
  \left\{
  \begin{aligned}
    &0 = \inf_{a\in A}\{\mathscr{L}^av^*(t,x) + f^a(t,x)\},\quad \forall (t,x)\in[0,T)\times\mathbb{R}^n,\\
    &v^*(T,x) = g(x),\quad \forall x\in\mathbb{R}^n,
  \end{aligned}
\right.
\end{equation}
which is exactly the Hamilton-Jacobi-Bellman (HJB) equation. 

Following the similar arguments of Eq.~\eqref{eq:DPP}--\eqref{eq:HJB} leads to the following linear Cuachy problem for the value function
\begin{equation}
  \label{eq:PDE-characterization}
  \left\{
  \begin{aligned}
    &0 = \mathscr{L}^\alpha v^\alpha(t,x) + f^\alpha(t,x),\quad \forall (t,x)\in[0,T)\times\mathbb{R}^n,\\
    &v^\alpha(T,x) = g(x),\quad \forall x\in\mathbb{R}^n,
  \end{aligned}
\right.
\end{equation}
where the infimum is absent because this value function might be not optimal. We refer to this as the PDE characterization of value functions.

\subsection{Characterizing value functions via FBSDEs}
As a result of Feynman-Kac's formula, solutions to PDEs~\eqref{eq:PDE-characterization} admit FBSDEs representation, and therefore it is possible to  characterized value functions with FBSDEs. To see this, one may apply It\^o's rule to find that
\begin{align*}
  dv^\alpha(s,X_s^{\alpha}) = \mathscr{L}^\alpha v^\alpha(s,X_s^{\alpha})\,ds + \langle \sigma^\intercal\partial_xv^\alpha(s,X_s^{\alpha}), dW_s\rangle.
\end{align*}
Substituting Eq.~\eqref{eq:PDE-characterization} into this equality and combining Eq.~\eqref{eq:general-SDE} yield the FBSDE characterization of $v^\alpha$
\begin{equation}
  \label{eq:FBSDE-characterization}
  \left\{
  \begin{aligned}
    X_s &= x + \int_t^sb^\alpha(\tau,X_\tau)\,d\tau + \int_t^s\sigma(\tau,X_\tau)\,dW_\tau,\\
    Y_s &= g(X_T) + \int_s^Tf^\alpha(\tau,X_\tau)\,d\tau - \int_s^T\langle Z_\tau,dW_\tau\rangle,\\
    Y_s &= v^\alpha(s,X_s),\quad \forall s\in[t,T],\quad d\mathbb{P}\text{-a.s.}, \\
    Z_s &= \sigma^\intercal\partial_xv^\alpha(s,X_s),\quad ds\otimes d\mathbb{P}\text{-a.e. on } [t,T]\times\Omega
  \end{aligned}
\right.
\end{equation}
under some conditions ensuring the solution's existence and uniqueness. We shall point out that this FBSDE is not in the most general form. In Eq.~\eqref{eq:FBSDE-characterization}, the forward SDE does not contain the backward part $Y$ as well as the control part $Z$. This means that the FBSDE is decoupled and we can separately solve the forward SDE and the backward SDE.

The PDE characterization Eq.~\eqref{eq:PDE-characterization} and FBSDE characterization Eq.~\eqref{eq:FBSDE-characterization}, along with HJB Eq.~\eqref{eq:HJB}, are fundamental motivations of this paper. However, in deriving these equations, we implicitly assume that $v^*$ and $v^\alpha$ are sufficiently smooth. This is nontrivial, especially for HJB Eq.~\eqref{eq:HJB}, which is strongly nonlinear. Nevertheless, we focus on problems such that this assumption holds, as the nonsmooth solution to HJB equation is already a broad topic, in which the concept of viscosity solutions must be introduced \cite{Crandall1983}. Extensions to the nonsmooth case might be considered in future works.

To conclude this section, we point out that the HJB Eq.~\eqref{eq:HJB} characterizing the optimal value function is a nonlinear PDE, while its reduced form Eq.~\eqref{eq:PDE-characterization}, satisfied by the value function of a given policy, is linear. From this point of view, the standard PI manages to approximate the solution to a nonlinear PDE with a sequence of solutions to linear PDEs. This linearization coincides with the idea of Newton's method for finding zeros, regarding some abstract arguments of general derivatives. However, as discussed in the last section, solving PDEs directly generally suffers the curse of dimensionality, and thus prevents applications in large-scale problems. This is the reason why we need the probabilistic formulation Eq.~\eqref{eq:FBSDE-characterization}.

\section{The PDE-based Policy Iteration Algorithm}\label{sec:PDE-subroutine}

In this section, we reformulate the system dynamics, state our assumptions, and present a global linear convergence result of the standard policy iteration algorithm. At last, we highlight two key issues with this PDE-based algorithm.

\subsection{Problem reformulation and assumptions}\label{sec:reformulation}
In this paper, we consider a slightly different system description other than the general form Eq.~\eqref{eq:general-SDE}. Specifically, we require the drift coefficient can be decomposed in a way such that the control-dependent term is explicitly coupled with the diffusion coefficient: $\forall (t,x,a)\in[0,T]\times\mathbb{R}^n\times A$,
\begin{equation}
  b(t,x,a) = \bar{b}(t,x) + \sigma(t,x)\hat{b}(t,x,a).
\end{equation}
Namely, $b(t,x,a)$ can be split into two parts; one $\bar{b}(t,x)$ is independent of control, and the other one $\sigma(t,x)\hat{b}(t,x,a)$ is control-dependent. It seems too restrictive at the first glance. But, if $\sigma\sigma^\intercal$ is nondegenerate, i.e., $(\sigma\sigma^\intercal)^{-1}$ exists on $[0,T]\times\mathbb{R}^n$, then the desired decomposition exists. Indeed, we can choose $\bar{b}\equiv0$ and $\hat{b}\equiv\sigma^\intercal(\sigma\sigma^\intercal)^{-1}b$. Also, we require that a measurable minimizing function $\mu$ is given such that for any $(t,x,z)\in[0,T]\times\mathbb{R}^n\times\mathbb{R}^d$,
\begin{equation}
  \label{eq:def-mu}
  \mu(t,x,z)\in \operatorname*{arginf}_{a\in A}\bigl\{\langle \hat{b}^a(t,x), z \rangle + f^a(t,x) \bigr\}.
\end{equation}
This function is useful in canceling the painful infimum operator in HJB equation. To see this, we note that the diffusion coefficient $\sigma$ is independent of control, and thus, for any $(t,x)$ and smooth function $v(\cdot,\cdot)$,
\begin{align*}
  &\hphantom{\,=\,}\operatorname*{arginf}_{a\in A}\{\mathscr{L}^av(t,x)+f^a(t,x)\} \\
&=
\operatorname*{arginf}_{a\in A}\{\langle \bar{b} + \sigma\hat{b}^a, \partial_xv(t,x) \rangle + f^a(t,x)\}\\
&= \mu(t,x,\sigma^\intercal\partial_xv(t,x)).
\end{align*}
We should stress that this property holds only for $b=\bar{b}+\sigma\hat{b}$. Without the explicit appearance of $\sigma\hat{b}$, the definition of $\mu$ would be problematic. However, for the affine system and quadratic control cost, which is main topic of adaptive dynamic programming \cite{lewis_vrabie_2009,Wang2009,theodorou2010,Song2014,Jiang2015,ISOCPI}, the minimizer of the right-hand side of Eq.~\eqref{eq:def-mu} uniquely exists and admits a closed analytic form. In particular, suppose that $b$ is linear in $a$ (then so is $\hat{b}$) and that $f$ is quadratic in $a$, and that $A$ is closed and convex. Then, $\mu$ can be obtained by projecting the minimizer of a quadratic function onto a closed convex set. See also \cite{Kerimkulov2020} for a more general discussion on the existence of $\mu$.

In order to rigorously state our algorithm and establish the desired convergence results, we need to pose some conditions on our problem. At first, we recall the useful uniform Lipschitz continuity and uniform boundness, which are able to ensure the existence and uniqueness of solutions to SDEs and BSDEs.
\begin{definition}[Uniform Lipschitz continuity and boundness]
    \label{def:uniform-Lipschitz-boundness}
    A continuous function $\phi(t,x,y)$ is said to be uniformly Lipschitz continuous in $x,y$ with respect to $t$ if there exists a positive constant $L$ such that for any $t\in E^1,\ x,x'\in E^2, \ y,y'\in E^3$,
    \begin{equation}
      \label{eq:uniform-Lipschitz}
      \|\phi(t,x,y) - \phi(t,x',y')\| \leq L \|x - x' \| + L\|y - y' \|,
    \end{equation}
    where $E^1, E^2, E^3$ are nonempty subsets of Euclidean spaces with proper dimensions. 

    Further, $\phi$ is said to be uniformly bounded if there exists a constant $L$ such that (suppose $0\in E^2,E^3$)
    \begin{equation}
      \label{eq:def-uniform-boundness}
      \|\phi(t,0,0)\| \leq L,\quad\forall t\in E^1.
    \end{equation}
    
    For convenience, let $C^{\mathrm{UniLip}}(E^1\times E^2\times E^3)$ denote the collection of functions satisfying Eq.~\eqref{eq:uniform-Lipschitz}, and $C^{\mathrm{UniLip}}_b(E^1\times E^2\times E^3)$ denote the collection of functions satisfying both Eq.~\eqref{eq:uniform-Lipschitz} and Eq.~\eqref{eq:def-uniform-boundness}.
\end{definition}
\begin{remark}
  Note that the Lipschitz condition may be only local if these subsets $E^1,E^2,E^3$ are bounded. In addition, any function valued in a bounded set is uniform bounded.

  A useful property of $C^{\mathrm{UniLip}}_b$ is the linear growth rate. Using the triangle inequality, any $\phi\in C^{\mathrm{UniLip}}_b(E^1,E^2,E^3)$ satisfies that, for any $(t,x,y)\in E^1\times E^2\times E^3$,
  \begin{equation*}
    \|\phi(t,x,y)\| \leq L(1 + \|x\| + \|y\|).
  \end{equation*}
\end{remark}

\begin{convention}
  For continuous function $\phi^1(t,x)$ or $\phi^2(x)$, we mean $\phi^1$ or $\phi^2\in C^{\mathrm{UniLip}}(E^1\times E^2\times E^3)$ if the extended function $\tilde{\phi}^1$ or $\tilde{\phi}^2\in C^{\mathrm{UniLip}}(E^1\times E^2\times E^3)$, where $$\tilde{\phi}^1(t,x,\cdot)\equiv\phi^1(t,x),\quad \tilde{\phi}^2(\cdot,x,\cdot)\equiv\phi^2(x).$$ We apply this simplification to $C^{\mathrm{UniLip}}_b$ too. 
\end{convention}

\begin{remark}
  For univariate function $\phi(\cdot)\in C^{\mathrm{UniLip}}$, the uniform boundness condition trivially holds.
\end{remark}

\begin{assumption}
  \label{assumption:1}
  Let the following assumptions hold.
  \begin{enumerate}
  \item The functions $\bar{b},\hat{b},\sigma,f, g\in C^{\mathrm{UniLip}}([0,T]\times\mathbb{R}^n\times A)$. Moreover, the given minimizing function $\mu\in C^{\mathrm{UniLip}}([0,T]\times\mathbb{R}^n\times \mathbb{R}^d)$.
  \item The functions $\bar{b},\mu,f$ are uniformly bounded: $\forall t\in[0,T]$,
    $$ \|\bar{b}(t,0,0)\| + \|\mu(t,0,0)\| + |f(t,0,0)| \leq L;$$
    and $\hat{b},\sigma$ are bounded: $\forall (t,x,a)\in[0,T]\times\mathbb{R}^n\times A$,
    \begin{equation*}
      \|\hat{b}(t,x,a)\| + \|\sigma(t,x)\|\leq L.
    \end{equation*}
  \item For any $\alpha\in C_b^{\mathrm{UniLip}}([0,T]\times\mathbb{R}^n)$, the linear Cauchy problem Eq.~\eqref{eq:PDE-characterization} has a smooth solution $w^\alpha\in C^{1,2}([0,T]\times\mathbb{R}^n)$ such that $\partial_xw^\alpha\in C^{\mathrm{UniLip}}_b$. Moreover, the HJB Eq.~\eqref{eq:HJB} has such a smooth solution $v^*$ too.
  \end{enumerate}
\end{assumption}
\begin{remark}
As a matter of fact, one essential condition on the existence of a smooth solution to HJB equation \eqref{eq:HJB} is the uniform elliptic condition: $\exists\delta > 0$ such that $y^\intercal\sigma\sigma^\intercal y\geq \delta y^\intercal y$ holds for any $(t,x,y)\in[0,T]\times\mathbb{R}^n\times\mathbb{R}^n$. Clearly, this condition is also sufficient to ensure the existence of $\bar{b}$ and $\hat{b}$.  
\end{remark}
\begin{remark}
  \label{remark:admissibility}
  Under Assumption~\ref{assumption:1}.1 and Assumption~\ref{assumption:1}.2, we have $b,\sigma,f,g,\mu\in C^{\mathrm{UniLip}}_b$, and thus, for any policy $\alpha\in C_b^{\mathrm{UniLip}}$ taking values in $A$, there is $b^\alpha\in C_b^{\mathrm{UniLip}}$. Hence, the solution to Eq.~\eqref{eq:general-SDE} uniquely exists for any $(t,x)$. Moreover, for any $\ell > 1$, $ \operatorname{\mathbb{E}}[\sup_{t\leq s\leq T}\|X^{\alpha,t,x}_s\|^\ell]$ is finite\cite{Liptser1977}.
\end{remark}
\begin{remark}
In linear quadratic problems, the assumptions that $f$ and $g$ are bounded and Lipschitz are violated. In practice, however, we can make some minor modifications to the problem in order to satisfy these assumptions. The idea is manually clipping the control and state in these functions below a certain threshold. For example, if $f(t,x,a)=x^\intercal Qx + a^\intercal Ra$, then  $\tilde{f}(t,x,a)=f(t,\tilde{x},\tilde{a})$ may be used, where $\tilde{x}$ and $\tilde{a}$ are component-wise clipped versions of $x$ and $a$, respectively. By choosing a sufficiently large threshold, we can still obtain a satisfactory suboptimal control policy of the original problem.
\end{remark}

We stress that Assumption~\ref{assumption:1} might not be the most general condition to make above assertions. But, it is very convenient to illustrate our key ideas without getting too involved into abstract theories of PDEs and SDEs. In particular, we have the following lemma to characterize value functions, which also serves as a starting point for the following subsections.

\begin{lemma}
  \label{lemma:value-function-characterization}
  Let Assumption~\ref{assumption:1} hold. Then, for any policy $\alpha\in C_b^{\mathrm{UniLip}}([0,T]\times\mathbb{R}^n)$ valued in $A$, the value function $v^\alpha$, defined by Eq.~\eqref{eq:general-SDE} and Eq.~\eqref{eq:def-valuefunction} is a unique solution to PDE~\eqref{eq:PDE-characterization} with $v^\alpha\in C^{1,2}$. Moreover, $v^\alpha$ admits the stochastic representation Eq.~\eqref{eq:FBSDE-characterization}.
\end{lemma}
\begin{proof}
  This is a direct consequence of Remark~\ref{remark:admissibility} and \cite[Theorem~7.4.1]{yong_zhou_1999}.
\end{proof}
\begin{remark}
  \label{remark:chain-admissibility}
  Under Assumption~\ref{assumption:1}.3, $\partial_xv^\alpha\in C^{\mathrm{UniLip}}_b$, and thus,  $\mu(\cdot,\cdot,\sigma^\intercal\partial_xv^\alpha(\cdot,\cdot))$ is a policy valued in $A$ and lies in $C^{\mathrm{UniLip}}$. It is also important to note that in the stochastic representation Eq~\eqref{eq:FBSDE-characterization}, the term $\sigma^\intercal\partial_xv^\alpha$ is encoded in the $Z$ process. Theorefore, obtaining $Z$ is to some extent sufficient to construct the policy $\mu(\cdot, \cdot, \sigma^\intercal\partial_xv^\alpha(\cdot,\cdot))$.
\end{remark}

\subsection{The standard policy iteration subroutine}

Let us focus on the HJB Eq.~\eqref{eq:HJB} and the PDE characterization Eq.~\eqref{eq:PDE-characterization}. Suppose $\alpha$ is an optimal policy, then $v^\alpha$ satisfies both of these equations. Combining Eq.~\eqref{eq:HJB} and Eq.~\eqref{eq:PDE-characterization}, for any $ (t,x)\in[0,T]\times\mathbb{R}^n$, we have
\begin{equation}
  \label{eq:verification-theorem}
  \mathscr{L}^\alpha v^\alpha(t,x) + f^\alpha(t,x) = \inf_{a\in A}\{\mathscr{L}^av^\alpha(t,x) + f^a(t,x)\}.
\end{equation}
Conversely, if this equation is satisfied by some policy $\alpha$, then its value function $v^\alpha$ satisfies the HJB equation. Hence, the central idea of policy iteration is to force Eq.~\eqref{eq:verification-theorem} to hold. 

The standard policy iteration algorithm works as follows.
\begin{enumerate}
\item Given a policy $\alpha$, find its value function $v^\alpha$ by Eq.~\eqref{eq:PDE-characterization}.
\item Given $v^\alpha$, find a policy $\alpha'$ such that for any $(t,x)$,
\begin{equation*}
  \alpha'(t,x) = \operatorname*{arginf}_{a\in A} \{\mathscr{L}^a v^\alpha(t,x) + f^a(t,x)\}.
\end{equation*}
\end{enumerate}
Alternatively repeating these two steps generates a sequence of policies. The first step is also known as policy evaluation, and the second step is policy improvement. According to Eq.~\eqref{eq:def-mu}, the policy improvement step can also be realized by setting
\begin{equation}
  \label{eq:improved-policy}
\alpha'(t,x) \coloneqq \mu(t,x,z(t,x)),\quad\forall (t,x)\in[0,T]\times\mathbb{R}^n,
\end{equation}
where $z(\cdot,\cdot)=\sigma^\intercal\partial_xv^\alpha(\cdot,\cdot)$. For simplicity, we combine policy evaluation and policy improvement into a single procedure and refer to it as the standard policy iteration subroutine, or the PDE-based subroutine; see Algorithm~\ref{alg:PDE-based-subroutine}. The global convergence result of GPI equipped with this subroutine is provided in Proposition~\ref{prop:PDE-convergence}.

\begin{algorithm}
\caption{A PDE-based subroutine of GPI.\label{alg:PDE-based-subroutine}}
\begin{algorithmic}[1]
\REQUIRE{a feedback control policy $\alpha$.}
\ENSURE{a feedback control policy $\alpha'$ not worse than $\alpha$.}
\STATE Obtain the value function $v^\alpha$ by Eq.~\eqref{eq:PDE-characterization}.
\STATE Construct the output policy by Eq.~\eqref{eq:improved-policy} with $z\leftarrow \sigma^\intercal\partial_xv^\alpha$.
\end{algorithmic}
\end{algorithm}

\begin{proposition}
  \label{prop:PDE-convergence}
  Let Assumption~\ref{assumption:1} hold. Starting at an initial policy $\alpha^0$ valued in $A$, let $\{\alpha_n\}_{n\in\mathbb{N}}$ denote the policy sequence generated by GPI equipped with Algorithm~\ref{alg:PDE-based-subroutine}. If $\alpha^0\in C_b^{\mathrm{UniLip}}([0,T]\times\mathbb{R}^n)$ is valued in $A$, then $\alpha^n$ is admissible for any $n\geq 0$. For any $(t,x)\in[0,T]\times\mathbb{R}^n$, the cost sequence $\{v^{\alpha^n}(t,x)\}_{n\in\mathbb{N}}$ is monotonically decreasing to $v^*(t,x)$. Moreover, there exists a constant $C=C(t,x)$ depending on $(t,x)$ and a constant $q\in(0,1)$ independent to $(t,x)$ such that
  \begin{equation}
    \label{eq:PDE-convergence}
    |v^{\alpha^n}(t,x) - v^*(t,x)| \leq C(t,x)q^n,\quad\text{for any } n \geq 0.
  \end{equation}
\end{proposition}
\begin{proof}
  The assertion of admissibility is a direct consequence of Remark~\ref{remark:chain-admissibility}. The monotonicity is also expected due to the definition of $\mu$\cite{ni_fang_2013}. Under our assumptions, Eq.~\eqref{eq:PDE-convergence} can be demonstrated by following the proof of \cite[Theorem~4.1]{Kerimkulov2020}, so we omit this technical proof here. The proof of Eq.~\eqref{eq:PDE-convergence} can also be viewed as a simplified version of the proof of Theorem~\ref{theorem:error-offpolicy}; see Remark~\ref{remark:error-offpolicy-proof-proposition-1} for more details.
\end{proof}

\subsection{Two key issues}\label{sec:two-key-issues}
To this end, we have formulated the PDE-based subroutine in Algorithm~\ref{alg:PDE-based-subroutine} and developed corresponding convergence results. Sadly, we have to admit that the global linear convergence rate in Proposition~\ref{prop:PDE-convergence} generally cannot be achieved with a practical program. The dilemma arises from the policy evaluation step.

The first issue is the design of numerical methods for policy evaluation. In Algorithm~\ref{alg:PDE-based-subroutine}, policy evaluation is formulated as solving PDEs, which generally has no closed form solution and has to be solved with numerical methods. Traditional numerical ways for PDEs require discretizing the time-state space, and thus, suffer from the curse of dimensionality. Moreover, extending traditional ways to model-free settings seems to be challenging. Based on these considerations, another two policy iteration subroutines utilizing the FBSDE characterization of value functions are proposed in Section~\ref{sec:FBSDE-subroutines}. We also develop a numerical method for solving FBSDEs by optimizing a novel criterion; see Section~\ref{sec:approximation}.

The second issue is more subtle. Since numerical methods cannot be expected to provide the exact solution, especially after time discretization, approximation errors are generally inevitable. Consequently, the improved policy based on this inexact solution is different from the expected output policy. To address this issue, we quantify these approximation errors as $\epsilon_n$ and analyze the convergence of the policy iteration with $\epsilon_n > 0$. We discuss this topic at the end of Section~\ref{sec:FBSDE-subroutines}.

\section{FBSDE-based Policy Iteration Algorithms}\label{sec:FBSDE-subroutines}

In this section, we propose two FBSDE-based policy iteration algorithms. The convergence result is established by showing the equivalence between the PDE-based and FBSDE-based policy iteration subroutines. At last, we present a robust convergence result with respect to approximation errors. In all following sections, the initial pair of time states $(t, x)$ is fixed.

\subsection{The on-policy subroutine}

In the PDE-based subroutine, the next trial policy is constructed by $\mu$ and $\sigma^\intercal\partial_xv^\alpha$, where the latter is obtained via solving PDE \eqref{eq:PDE-characterization}. In view of  Lemma~\ref{lemma:value-function-characterization}, it is very natural to consider carrying out policy evaluation by solving FBSDE \eqref{eq:FBSDE-characterization}. We formulate this idea in Algorithm~\ref{alg:on-policy-FBSDE}.

The second step of Algorithm~\ref{alg:on-policy-FBSDE} is the key of this work. Instead of evaluating $v^\alpha$ via a linear PDE and substituing $\partial_xv^\alpha$ into the policy improvement step, we directly obtain a $z^\alpha$ term via an optimization problem, and then construct the next trial policy based on it. We will discuss in detail how to minimize the objective function Eq.~\eqref{eq:on-policy-op} in Section~\ref{sec:approximation}. Here, we simply assume that there is a method that can be used to determine the global solution $z^\alpha$.

\begin{algorithm}
\caption{The on-policy subroutine of GPI.\label{alg:on-policy-FBSDE}}
\begin{algorithmic}[1]
\REQUIRE{a feedback control policy $\alpha$; an initial point $(t,x)$.}
\ENSURE{a feedback control policy $\alpha'$ not worse than $\alpha$.}
\STATE
Find the solution $X^\alpha$ to the forward SDE~\eqref{eq:general-SDE}.
\STATE
Find an optimal solution $z^\alpha$ to the optimization problem
\begin{equation}
  \label{eq:on-policy-op}
  \min_{z\in C^{\mathrm{UniLip}}_b} \epsilon^\alpha\coloneqq  \operatorname{\mathbb{E}}\int_t^T\|z(s,X^\alpha_s) - Z^\alpha_s\|^2\,ds,
\end{equation}
where $Z^\alpha$ is a part of the solution to the following BSDE
\begin{equation*}
    Y_s = g(X^\alpha_T) + \int_s^Tf^\alpha(\tau,X^\alpha_\tau)\,d\tau - \int_s^T \langle Z_\tau, dW_\tau\rangle.
\end{equation*}
\STATE
Construct the output policy by Eq.~\eqref{eq:improved-policy} with $z\leftarrow z^\alpha$.
\end{algorithmic}
\end{algorithm}

Comparing the policies returned by Algorithm~\ref{alg:on-policy-FBSDE} and Algorithm~\ref{alg:PDE-based-subroutine}, it can be seen that $z^\alpha$ plays the role of $\sigma^\intercal\partial_xv^\alpha$. According to Lemma~\ref{lemma:value-function-characterization}, $\sigma^\intercal\partial_xv^\alpha$ is indeed a global solution to that optimization problem. Noting that $Z^\alpha_s=\sigma^\intercal\partial_xv^\alpha(s,X_s^\alpha)$ holds almost everywhere on the product space $[t,T]\times\Omega$, we can rewrite the objective function Eq.~\eqref{eq:on-policy-op} as
\begin{equation}
\label{eq:def-error-h}
\epsilon^\alpha(z)  = \operatorname{\mathbb{E}}\int_t^Th(s,X_s^{\alpha})\,ds,
\end{equation}
where $h(\cdot,\cdot)\coloneqq \|z(\cdot,\cdot)-\sigma^\intercal\partial_xv^\alpha(\cdot,\cdot)\|^2\geq 0$. Hence, we have $\epsilon^\alpha(z^\alpha)=0$. In the opposite direction, however, one cannot say that $\sigma^\intercal\partial_xv^\alpha$ is the unique optimal solution in $C^{\mathrm{UniLip}}_b$, since $h\equiv0$ is not the necessary condition of $\epsilon^\alpha=0$. In fact, the necessary and suffcient condition is $h$ equals zero almost everywhere on the product space under the measure induced by $X^\alpha(s,\omega)$. To put it another way, we can only say that $z^\alpha(\cdot,\cdot)$ equals $\sigma^\intercal\partial_xv^\alpha(\cdot,\cdot)$ almost everywhere along the process $X^{\alpha,t,x}$. Fortunately, Lemma~\ref{lemma:change-of-policy} below suggests that this almost everywhere identity is enough to guarantee that Algorithm~\ref{alg:PDE-based-subroutine} and Algorithm~\ref{alg:on-policy-FBSDE} are equivalent, in the sense that the returned policies have the same cost value.

Before proceeding, we would like to clarify one more point regarding this algorithm. The first two steps for obtaining $z^\alpha$ can be implemented in a pure data-driven fashion. The forward state process $\{X_s^\alpha\}_{t\leq s\leq T}$ can be sampled by sending the current policy $\alpha$ to the dynamic system and observe the state trajectory. Furthermore, it is possible to solve that optimization problem using only samples without knowing the exact solution $(Y^\alpha,Z^\alpha)$. This is the reason why we call Algorithm~\ref{alg:on-policy-FBSDE} on-policy. In the next subsection, we introduce the off-policy subroutine, where the forward SDE is driven by a fixed behavior policy $\alpha^b$ instead of the current policy $\alpha$.

\begin{lemma}
  \label{lemma:change-of-policy}
  Let Assumption~\ref{assumption:1} hold. For any $\alpha^1,\alpha^2\in C_b^{\mathrm{UniLip}}([0,T]\times\mathbb{R}^n)$, let $X^1,X^2$ be their state processes, respectively. Then, for any nonnegative measurable function $h(\cdot,\cdot)\geq0$, the following statements are equivalent:
  \begin{enumerate}
  \item $h(s,X_s^1)=0$ holds $ds\otimes d\mathbb{P}$-a.e. on $[t,T]\times\Omega$;
  \item $h(s,X_s^2)=0$ holds $ds\otimes d\mathbb{P}$-a.e. on $[t,T]\times\Omega$.
  \end{enumerate}
\end{lemma}
\begin{proof}
  Consider the following two auxiliary processes
  \begin{align*}
    W_s^i &= W_s + \int_t^s\hat{b}^{\alpha^i}(\tau,X_\tau^i)\,d\tau,\quad s\in[t,T],\quad i=1,2.
  \end{align*}
  Noting that $\{\hat{b}^{\alpha^i}(s,X_s^i);t\leq s\leq T\}$ is bounded and thus satisfies Novikov condition, there exists probability measure $\mathbb{P}^i$, equivalent to $\mathbb{P}$, such that $W^i$ becomes a standard Brownian motion under $\mathbb{P}^i$. This is known as the  Girsanov's theorem\cite[Chapter~3]{karatzas_shreve_1998}. Therefore, $(X^1,W^1,\mathbb{P}^1)$ and $(X^2,W^2,\mathbb{P}^2)$ are two weak solutions to the following SDE:
  \begin{equation*}
    X_s = x + \int_t^s\bar{b}(\tau,X_\tau)\,d\tau + \int_t^s\sigma(\tau,X_\tau)\,dW_\tau.
  \end{equation*}
  By the uniformly Lipschitz continuity and boundness of $\bar{b}$ and $\sigma$, the strong existence and uniqueness hold for this SDE. Then the weak uniqueness in the sense of probability law holds too, namely, $X^1$ and $X^2$ have the same law. Thus, the integral of $h(s,X^1_s)$ equals the integral of $h(s,X_s^2)$:
  \begin{equation*}
    \int_t^T\biggl(\int h(s,X_s^1)\,d\mathbb{P}^1\biggr)\,ds = \int_t^T\biggl(\int h(s,X_s^2)\,d\mathbb{P}^2\biggr)\,ds.
  \end{equation*}
  We conclude that $h(s,X_s^1)=0$ holds $ds\otimes\,d\mathbb{P}^1$-a.e. if and only if $h(s,X_s^2)=0$ holds $ds\otimes\,d\mathbb{P}^2$-a.e.. The proof is finished by noting that $\mathbb{P},\mathbb{P}^1,\mathbb{P}^2$ are equivalent to each other.
\end{proof}

This lemma offers the freedom of changing the underlying process in the optimization problem of the on-policy subroutine. By setting $h(\cdot,\cdot)$ as Eq.~\eqref{eq:def-error-h}, this lemma suggests that minimizing the $ \operatorname{\mathbb{E}}\int_t^Th(s,X_s^{\alpha})$ to zero is equivalent to minimizing $ \operatorname{\mathbb{E}}\int_t^Th(s,X_s^{\alpha^b})$ to zero for any $\alpha^b\in C^{\mathrm{UniLip}}_b$. Thus, it is also reasonable to choose a policy $\alpha^b$ different from $\alpha$ and optimize the integral of $h$ along $X^{\alpha^b}$. On the other hand, let $ \operatorname{\mathbb{E}}\int_t^Th(s,X_s^{\alpha})=0$ hold and $\alpha'$ be the policy returned by Algorithm~\ref{alg:PDE-based-subroutine}. Then, $h(s,X_s^{\alpha'})=0$ holds almost everywhere on the product space $[t,T]\times\Omega$. This argument is also the key to prove the following equivalence between the PDE-based subroutine and the on-policy subroutine.

\begin{theorem}
  \label{theorem:equivalence-PDE-FBSDE}
  Let Assumption~\ref{assumption:1} hold. For an input policy $\alpha\in C_b^{\mathrm{UniLip}}([0,T]\times\mathbb{R}^n)$ valued in $A$, let $\alpha_{1}'$ and $\alpha_{2}'$ denote the outputs of Algorithm~\ref{alg:PDE-based-subroutine} and Algorithm~\ref{alg:on-policy-FBSDE}, respectively. Then, $\alpha'_{1}$ and $\alpha'_{2}$ generate the ``same'' trajectory starting at $(t,x)$:
  \begin{equation*}
    X_s^{\alpha'_{1},t,x} = X_s^{\alpha'_{2},t,x},\quad ds\otimes d\mathbb{P}\text{-a.e. on }[t,T]\times\Omega.
  \end{equation*}
Moreover, $v^{\alpha'_{1}}(t,x) = v^{\alpha'_{2}}(t,x)$.
\end{theorem}
\begin{proof}
  Let $v^\alpha$ and $z^\alpha$ denote the same objects in Algorithm \ref{alg:PDE-based-subroutine} and Algorithm~\ref{alg:on-policy-FBSDE},  repsectively. We write down the explicit expression of $\alpha'_1,\alpha'_2$:
  \begin{equation*}
    \alpha'_1(\cdot,\cdot) = \mu(\cdot,\cdot,\sigma^\intercal\partial_xv^\alpha(\cdot,\cdot)),~~\alpha'_2(\cdot,\cdot) = \mu(\cdot,\cdot,z^\alpha(\cdot,\cdot)),
  \end{equation*}
and denote by $h(\cdot,\cdot)=\|z^\alpha(\cdot,\cdot) - \sigma^\intercal\partial_xv^\alpha(\cdot,\cdot)\|^2$. 

  According to Remark~\ref{remark:chain-admissibility}, $\alpha'_1,\alpha'_2$ are admissible. Consider the forward SDEs satisfied by $X^{\alpha'_1},X^{\alpha'_2}$:
\begin{equation*}
  \begin{aligned}
    X^{\alpha'_1}_s &= x + \int_t^sb^{\alpha'_1}(\tau,X^{\alpha'_1}_\tau)\,d\tau + \int_t^s\sigma(\tau,X^{\alpha'_1}_\tau)d W_s,\\
    X^{\alpha'_2}_s &= x + \int_t^sb^{\alpha'_2}(\tau,X^{\alpha'_2}_\tau)\,d\tau + \int_t^s\sigma(\tau,X^{\alpha'_2}_\tau)d W_s.
  \end{aligned}
\end{equation*}
  We claim that 
\begin{equation}
  \label{eq:tmp-378}
\alpha'_1(s,X_s^{\alpha'_1})=\alpha'_2(s,X_s^{\alpha'_1}),\quad ds\otimes d\mathbb{P}\text{-a.e. on }[t,T].
\end{equation}
Indeed, it can be concluded from  Lemma~\ref{lemma:value-function-characterization} that $h(s,X_s^\alpha)=0$ almost everywhere on $[t,T]\times\Omega$. Then, applying Lemma~\ref{lemma:change-of-policy} yields $h(s,X_s^{\alpha'_1})=0$ almost everywhere. Denote by 
$$\widetilde{X}_s^{\alpha_1'} = x + \int_t^sb^{\alpha'_2}(\tau,X^{\alpha'_1}_\tau)\,d\tau + \int_t^s\sigma(\tau,X^{\alpha'_1}_\tau)d W_s$$
and $\phi(u)\coloneqq \operatorname{\mathbb{E}}\int_t^{t+u}\|X^{\alpha'_1}_\tau - X^{\alpha'_2}_\tau\|^2\,d\tau$. Noting Eq.~\eqref{eq:tmp-378} and the Lipschitz continuity of $b^{\alpha'_2}$ and $\sigma$, we have
\begin{align*}
\phi(u) &= \operatorname{\mathbb{E}}\int_t^{t+u}\|\widetilde{X}^{\alpha'_1}_s - X^{\alpha'_2}_s\|^2\,ds\\
&\leq \operatorname{\mathbb{E}} \int_t^{t+u}\biggl\{  2\biggl[ \int_t^s(b^{\alpha'_2}(\tau,X_\tau^{\alpha'_1}) - b^{\alpha'_2}(\tau,X_\tau^{\alpha'_2}) )\,d\tau\biggr]^2 \\
& \hphantom{\,\leq \operatorname{\mathbb{E}}} + 2 \biggl[ \int_t^s(\sigma(\tau,X_\tau^{\alpha'_1}) - \sigma(\tau,X_\tau^{\alpha'_2}) )\,dW_\tau\biggr]^2  \biggr\}\,ds \\
&\leq \operatorname{\mathbb{E}}\int_t^{t+u} 2(s-t+1)L^2\int_t^s\|X_\tau^{\alpha'_1} - X_\tau^{\alpha'_2}\|^2\,d\tau \,ds \\
&\leq 2(T+1)L^2\int_0^u\phi(s)\,ds,\quad \forall u\in[0,T-t].
\end{align*}
Hence, by Gr\"onwall's inequality, there is $\phi(T-t)=0$. This proves that $X_s^{\alpha'_1}=X_s^{\alpha'_2}$ almost everywhere on $[t,T]\times\Omega$. Moreover, the cost of $\alpha'_1$ and $\alpha'_2$ at $(t,x)$ is equal.
\end{proof}
\begin{remark}
  This result reveals that there is no difference between the cost sequence produced by GPI using the PDE-based subroutine and the on-policy subroutine. Thus, all the convergence properties of the standard PI is preserved in our probabilistic framework.
\end{remark}

\begin{corollary}
  \label{corollary:FBSDE-convergence}
  For any fixed $(t,x)\in[0,T]\times\mathbb{R}^n$, the conclusions of Proposition~\ref{prop:PDE-convergence} hold if Algorithm~\ref{alg:PDE-based-subroutine} is replaced  by Algorithm~\ref{alg:on-policy-FBSDE}.
\end{corollary}
\begin{remark}
  Because the output of Algorithm~\ref{alg:on-policy-FBSDE} may depend on the argument $(t,x)$, we cannot make a conclusion that $\{v^{\alpha^n}(t',x')\}$ is monotone at any $(t',x')$ as in Proposition~\ref{prop:PDE-convergence}. Nevertheless, the cost sequence $\{v^{\alpha^n}(t,x)\}$ is still monotonically decreasing, where $(t,x)$ is the argument passed into Algorithm~\ref{alg:on-policy-FBSDE}.
\end{remark}

\subsection{The off-policy subroutine}\label{sec:BSDE-subroutine}
On-policy and off-policy are terminologies in reinforcement learning \cite{Sutton2018}. Roughly speaking, on-policy and off-policy algorithms are both data-driven but different in the way of collecting data. In an on-policy algorithm, a value function of a policy $\alpha$ is evaluated with data collected by itself. This corresponds to FBSDE \eqref{eq:FBSDE-characterization}, where the forward SDE is driven by $\alpha$ and the solution to the backward SDE is related to $v^\alpha$ too. However, in an off-policy algorithm, the value function $v^\alpha$ is generally evaluated with data collected by a different policy, called the behavior policy $\alpha^b$ usually. The advantage of off-policy algorithms is the high data efficient. If we adopt the on-policy subroutine Algorithm~\ref{alg:on-policy-FBSDE} in GPI, then the current policy $\alpha$ generally changes during the iteration. Therefore, we have to resample data at the beginning of each iteration, i.e., solving a new forward SDE in our case. On the other hand, if we adopt the off-policy technique, then we can evaluate the value function of the new policy with pre-collected data. This in turn improves the data efficiency. Also, we expect that the variance effect would be reduced because all data are collected by the same policy.

With the help of nonlinear Feynman-Kac's formula, it is straightforward to extend the on-policy FBSDE characterization of value function to the off-policy case.

\begin{lemma}
  \label{lemma:off-policy-characterization}
  Let the condition of Lemma~\ref{lemma:value-function-characterization} hold and use the same notation. For any policy $\alpha^b\in C_b^{\mathrm{UniLip}}([0,T]\times\mathbb{R}^n)$ valued in $A$, the value function $v^\alpha$ admits the following stochastic representation:
\begin{equation}
  \label{eq:off-policy-characterization}
  \left\{
  \begin{aligned}
    X^b_s &= x + \int_t^sb^{\alpha^b}(\tau,X^b_\tau)\,d\tau + \int_t^s\sigma(\tau,X^b_\tau)\,dW_\tau,\\
    Y_s &= g(X^b_T) + \int_s^T f^\alpha(\tau,X^b_\tau)\,d\tau - \int_s^T\langle Z_\tau,dW_\tau\rangle\\
    & \hphantom{\,=} + \int_s^T\langle \hat{b}^{\alpha}(\tau,X^b_\tau)-\hat{b}^{\alpha^b}(\tau,X^b_\tau), Z_\tau\rangle\,d\tau, \\
    Y_s &= v^\alpha(s,X^b_s),\quad \forall s\in[t,T],\quad d\mathbb{P}\text{-a.s.}, \\
    Z_s &= \sigma^\intercal\partial_xv^\alpha(s,X^b_s),\quad ds\otimes d\mathbb{P}\text{-a.e. on } [t,T]\times\Omega.
  \end{aligned}
\right.
\end{equation}
\end{lemma}
\begin{proof}
  By the definitions of $\bar{b}, \hat{b}$ and $\mu$, we can rewrite the PDE satisfied by $v^\alpha$ as follows:
\begin{equation*}
  \left\{
  \begin{aligned}
  &0 = \langle \hat{b}^\alpha(t,x)-\hat{b}^{\alpha^b}(t,x),\sigma^\intercal\partial_xv^\alpha(t,x)\rangle \\
    &\hphantom{0 =,} + \mathscr{L}^{\alpha^b}v^\alpha(t,x) + f^\alpha(t,x) ,\quad \forall (t,x)\in[0,T)\times\mathbb{R}^n,\\
    &v^\alpha(T,x) = g(x),\quad\forall x\in\mathbb{R}^n.
  \end{aligned}
\right.
\end{equation*}
Applying the nonlinear Feynman-Kac's formula\cite[Theorem~7.4.5]{yong_zhou_1999} to this leads to the desired representation. 
\end{proof}
\begin{remark}
  If $\alpha^b\equiv \alpha$, this degenerates to Lemma~\ref{lemma:value-function-characterization}. It is important to note that the forward process $X^b$ is independent of $\alpha$, which is the key difference between the on-policy and off-policy methods. GPI equipped with the off-policy subroutine and a fixed $\alpha^b$ should be viewed as an iteration of BSDEs, while that equipped with the on-policy subroutine should be viewed as an iteration of FBSDEs.
\end{remark}

Based on Lemma~\ref{lemma:off-policy-characterization}, we propose Algorithm~\ref{alg:off-policy-BSDE}, in which the optimization problem is modified according to Eq.~\eqref{eq:off-policy-characterization}. It can be concluded from Lemma~\ref{lemma:change-of-policy} that the optimization problems in Algorithm~\ref{alg:on-policy-FBSDE} and Algorithm~\ref{alg:off-policy-BSDE} have the same solutions. In light of this observation, we are able to prove that the returned polices of on-policy and off-policy subroutines are equivalent.

\begin{algorithm}
\caption{The off-policy subroutine of GPI.\label{alg:off-policy-BSDE}}
\begin{algorithmic}[1]
\REQUIRE{policies $\alpha,\alpha^b$; an initial condition $(t,x)$.}
\ENSURE{a policy $\alpha'$ not worse than $\alpha$.}
\STATE
Find the solution $X^{b}$ to the forward SDE~\eqref{eq:general-SDE} with $\alpha\leftarrow\alpha^b$.
\STATE
Find an optimal solution $z^\alpha$ to the optimization problem
\begin{equation}
  \label{eq:off-policy-op}
  \min_{z\in C^{\mathrm{UniLip}}_b} \epsilon^\alpha\coloneqq  \operatorname{\mathbb{E}}\int_t^T\|z(s,X^b_s) - Z^{\alpha,b}_s\|^2\,ds,
\end{equation}
where $Z^{\alpha,b}$ is a part of the solution to the following BSDE
\begin{align*}
    Y_s^{\alpha,b} &= g(X^b_T) + \int_s^T f^\alpha(\tau,X^b_\tau)\,d\tau - \int_s^T\langle Z^{\alpha,b}_\tau,dW_\tau\rangle\\
    & \hphantom{\,=} + \int_s^T\langle \hat{b}^{\alpha}(\tau,X^b_\tau)-\hat{b}^{\alpha^b}(\tau,X^b_\tau), Z^{\alpha,b}_\tau\rangle\,d\tau.
\end{align*}
\STATE
Construct the output policy by Eq.~\eqref{eq:improved-policy} with $z\leftarrow z^\alpha$.
\end{algorithmic}
\end{algorithm}

\begin{theorem}
  \label{theorem:equivalence-on-policy-off-policy}
Let the condition of Theorem~\ref{theorem:equivalence-PDE-FBSDE} hold and use the same notation. If $\alpha^b\in C_b^{\mathrm{UniLip}}([0,T]\times\mathbb{R}^n)$ is valued in $A$, then the output policies of Algorithm~\ref{alg:on-policy-FBSDE}, ~\ref{alg:off-policy-BSDE}, denoted by $\alpha'_2, \alpha'_3$, generate the ``same'' trajectory starting at $(t,x)$:
  \begin{equation*}
    X_s^{\alpha'_{2},t,x} = X_s^{\alpha'_{3},t,x},\quad ds\otimes d\mathbb{P}\text{-a.e. on }[t,T]\times\Omega.
  \end{equation*}
Moreover, $v^{\alpha'_{2}}(t,x) = v^{\alpha'_{3}}(t,x)$.
\end{theorem}
\begin{proof}
  The proof is similar to the proof of Theorem~\ref{theorem:equivalence-PDE-FBSDE} except that we need to show
\begin{equation*}
    \alpha'_2(s,X_s^{\alpha'_2}) = \alpha'_3(s,X_s^{\alpha'_2}),\quad ds\otimes d\mathbb{P}\text{-a.e. on }[t,T]\times\Omega.
  \end{equation*}
Let $z^\alpha_i (i=2,3)$ be the term $z^\alpha$ in Algorithm~\ref{alg:on-policy-FBSDE} and Algorithm~\ref{alg:off-policy-BSDE}, respectively. Using Lemma~\ref{lemma:value-function-characterization}--\ref{lemma:off-policy-characterization}, we have 
\begin{equation*}
z_i^\alpha(s,X_s^{\alpha'_2}) = \sigma^\intercal\partial_xv^\alpha(s,X_s^{\alpha'_2}),\quad ds\otimes d\mathbb{P}\text{-a.e. on }[t,T]\times\Omega.
\end{equation*}
Substituting this into the definition of $\alpha'_i$ finishes our proof.
\end{proof}

In view of Theorem~\ref{theorem:equivalence-PDE-FBSDE} and Theorem~\ref{theorem:equivalence-on-policy-off-policy}, we conclude that  these three subroutines are equivalent to each other. Consequently, the following convergence result for Algorithm~\ref{alg:off-policy-BSDE} holds.

\begin{corollary}
  \label{corollary:BSDE-convergence}
  For any fixed $(t,x)\in[0,T]\times\mathbb{R}^n$ and $\alpha^b\in C^{\mathrm{UniLip}}([0,T]\times\mathbb{R}^n)$ valued in $A$, the conclusions of Proposition~\ref{prop:PDE-convergence} hold if Algorithm~\ref{alg:PDE-based-subroutine} is replaced by Algorithm~\ref{alg:off-policy-BSDE}.
\end{corollary}

\subsection{A robust convergence result}\label{sec:robust-convergence-analysis}
Consider the optimization problem in Algorithm~\ref{alg:off-policy-BSDE}. In view of Lemma~\ref{lemma:off-policy-characterization}, there exists a $z^\alpha$ with $\epsilon^\alpha(z^\alpha)=0$. In practice, however, it is usually the case that we can only find a suboptimal solution $\hat{z}$ and thus $\epsilon^\alpha(\hat{z}) > 0$. If we construct a policy by Eq.~\eqref{eq:improved-policy} with $z\leftarrow \hat{z}$, then there is no guarantee that this new policy $\hat{\alpha}$ performs better than the current policy $\alpha$. To see this, we apply It\^o's formula to obtain (noting the PDEs satisfied by value functions)
\begin{align*}
&\hphantom{\,=\, }v^{\alpha}(t,x) - v^{\hat{\alpha}}(t,x) \\
&= \operatorname{\mathbb{E}} \int_t^T(\mathscr{L}^{\hat{\alpha}} v^{\hat{\alpha}} - \mathscr{L}^{\hat{\alpha}} v^{\alpha}) (s,X^{\hat{\alpha}}_s)\,ds \\
&= \operatorname{\mathbb{E}}\int_t^T\bigl(\mathscr{L}^{\alpha}v^{\alpha}+f^{\alpha} - \mathscr{L}^{\hat{\alpha}}v^{\alpha} - f^{\hat{\alpha}}\bigr)(s,X_s^{\hat{\alpha}})\,ds.
\end{align*}
If $\hat{z}=z^\alpha$, then $\hat{\alpha}(s,X_s^{\hat{\alpha}})=\mu(s,X_s^{\hat{\alpha}},\sigma^\intercal\partial_xv^\alpha(s,X_s^{\hat{\alpha}}))$ almost everywhere on $[t,T]\times\Omega$, and thus, $v^{\alpha}(t,x) - v^{\hat{\alpha}}(t,x)$ equals
$$\operatorname{\mathbb{E}}\int_t^T\bigl(\mathscr{L}^{\alpha}v^{\alpha}+f^{\alpha} - \inf_{a\in A}\{ \mathscr{L}^av^{\alpha} + f^a\}\bigr)(s,X_s^{\hat{\alpha}})\,ds \geq 0.$$
If $\epsilon^\alpha(\hat{z}) > 0$, then generally $v^\alpha(t,x)\geq v^{\hat{\alpha}}(t,x)$ does not hold, and thus the monotonicity of policy improvement is broken.

Below, we study the case in which the objective function in the off-policy subroutine does not reach zero during policy iterations. Though the cost sequence $\{v^{\alpha^n}(t,x)\}$ may be not monotone, we show that it still converges to the optimal cost if the $n$-th objective value $\epsilon_n$ converges to zero. To make it more clear, we spell the policy iteration procedure in Algorithm~\ref{alg:noised-off-PI}. In comparison to the GPI that is equipped with the off-policy subroutine, Algorithm~\ref{alg:noised-off-PI} contains two important differences. The first difference is that the behavior policy $\alpha^b$ is fixed during iteration. This is not the only way to apply the off-policy BSDE subroutine in GPI, as it can be proved that the cost of the output policy does not change if $\alpha^b$ is different. In order to view the whole algorithm as the iteration of BSDEs, however, we do not allow the forward SDE changes during iteration. The second difference is that $z^n$ is not necessarily an optimal solution of Eq.~\eqref{eq:off-policy-op}. Also, $\epsilon_n$ is not necessarily equal to 0.

\begin{algorithm}
\caption{A BSDE-based Policy Iteration Algorithm.\label{alg:noised-off-PI}}
\begin{algorithmic}[1]
\REQUIRE{policies $\alpha^0,\alpha^b$; an initial condition $(t,x)$.}
\ENSURE{a sequence of policies $\{\alpha^n\}$.}
\STATE
Find the solution $X^{b}$ to the forward SDE~\eqref{eq:general-SDE} with $\alpha\leftarrow\alpha^b$.
\FOR{$n=0,1,2,\ldots$}
\STATE
Run a numerical method to solve the optimization problem~\eqref{eq:off-policy-op} with $\alpha\leftarrow \alpha^n$. Denote by $z^n$ the returned solution and $\epsilon_n$ the associated objective value.
\STATE
Construct $\alpha^{n+1}$ by Eq.~\eqref{eq:improved-policy} with $z\leftarrow z^n$.
\ENDFOR
\end{algorithmic}
\end{algorithm}

With notations defined in Algorithm~\ref{alg:noised-off-PI}, we can state our robust convergence result as follows.
\begin{theorem}
  \label{theorem:error-offpolicy}
  Let Assumption~\ref{assumption:1} hold and use notations in   Algorithm~\ref{alg:noised-off-PI}. If $\alpha^0, \alpha^b\in C^{\mathrm{UniLip}}_b([0,T]\times\mathbb{R}^n)$ are policies valued in $A$, then $\alpha^n$ is admissible for any $n\geq 0$. Moreover, there exist constants $q\in(0,1)$ and $\gamma > 0$, both independent of $(t,x)$, such that the following inequality holds
\begin{equation*}
\operatorname*{limsup}_{n\to\infty}\bigl|v^{\alpha^n}(t,x) - v^*(t,x)\bigr|^2 \leq \frac{qe^{\gamma(T-t)}}{1-q}\cdot \operatorname*{limsup}_{n\to\infty}\epsilon_n.    
\end{equation*}
\end{theorem}
\begin{proof}
Throughout this proof, we fix the forward state to $X^b$, and use $F_s$ to denote $F(s,X_s^b)$ for any function $F(\cdot,\cdot)$.

The admissibility is a direct consequence of Remark~\ref{remark:admissibility}. According to Lemma~\ref{lemma:off-policy-characterization}, for $n\geq 1$, we have
\begin{equation*}
Y_s^n = g(X_T^b) + \int_s^T\kern-0.5em f^{\alpha^{n}}_\tau\!\! + \! (\hat{b}^{\alpha^{n}}_\tau-\hat{b}^{\alpha^b}_\tau)^\intercal Z^n_\tau\,d\tau  - \int_s^T\kern-0.5em(Z^n_\tau)^\intercal dW_\tau,
\end{equation*}
where $Y_s^n=v^{\alpha^n}(s,X_s^b),\  Z_s^n=\sigma^\intercal\partial_xv^{\alpha^n}(s,X_s^b)$. Similarly,
\begin{equation*}
  \begin{aligned}
Y_s^* = g(X_T^b) + \int_s^T\kern-0.5em f^{\alpha^{*}}_\tau\!\! + \! (\hat{b}^{\alpha^{*}}_\tau-\hat{b}^{\alpha^b}_\tau)^\intercal Z^*_\tau\,d\tau  - \int_s^T\kern-0.5em(Z^*_\tau)^\intercal dW_\tau,
  \end{aligned}
\end{equation*}
where $Y^*_s=v^*(s,X_s^b),\ Z^*_s = \sigma^\intercal\partial_xv^*(s,X_s^b)$.

Define $h:\Omega\times[0,T]\times\mathbb{R}^d\times\mathbb{R}^d\to\mathbb{R}$ by
\begin{equation*}
  h(s,z,Z)\coloneqq f^{\mu_s(z)}_s + \langle \hat{b}^{\mu_s(z)}_s - \hat{b}^{\alpha^b}_s, Z\rangle.
\end{equation*}
Then, we can verify that under Assumption~\ref{assumption:1} there is a  constant $L$ such that for any $(s,z,Z)\in[t,T]\times\mathbb{R}^d\times\mathbb{R}^d$, 
\begin{equation*}
  |h(s,z,Z) - h(s,0,0)|\leq L\|z\| + L\|Z\|,\quad \mathbb{P}\text{-a.s.},
\end{equation*}
and that $\operatorname{\mathbb{E}}\int_t^T\|h(s,0,0)\|^2\,ds < \infty$. Moreover, it can be proved that $\|Z^*_s\| \leq L\|\partial_xv^*(s,X_s^b)\|$ can be further bounded by some constant $K$ \cite[Proposition~4.3.1]{yong_zhou_1999}; see also \cite[Chapter~4]{Krylov2008} for more general discussions on the properties of $\partial_xv^\alpha$. Hence, we have
\begin{equation*}
  \begin{aligned}
    & \hphantom{\leq \ } |h(s,z^{n-1}_s,Z^n_s) - h(s,Z^*_s,Z^*_s)| \\
 &\leq |f^{\mu_s(z^{n-1}_s)}_s - f^{\mu_s(Z^{*}_s)}_s| + |\langle \hat{b}^{\mu_s(z^{n-1}_s)}_s - \hat{b}^{\alpha^b}_s, Z^n_s - Z^*_s\rangle| \\
&\hphantom{\leq\ } + |\langle \hat{b}^{\mu_s(z^{n-1}_s)}_s - \hat{b}^{\mu_s(Z^*_s)}_s, Z^*_s\rangle|\\
&\leq L\|\mu_s(z^{n-1}_s) - \mu_s(Z^{*}_s)\| + 2L\|Z_s^n - Z_s^*\| \\
&\hphantom{\leq\ } + K\|\hat{b}^{\mu_s(z^{n-1}_s)}_s - \hat{b}^{\mu_s(Z^*_s)}_s\|\\
&= (L^2+KL)\|z^{n-1}_s-Z^*_s\| + 2L\|Z_s^n-Z_s^*\|.
  \end{aligned}
\end{equation*}
To this end, all conditions of \cite[Lemma~A.5]{Kerimkulov2020} are verified, and thus, the following estimation holds for any $n\geq 1$:
\begin{equation}
  \label{eq:tmp-801}
  \begin{aligned}
  & \hphantom{\,\leq} \operatorname{\mathbb{E}}|Y_t^n - Y_t^*|^2 + \operatorname{\mathbb{E}}\int_t^Te^{\gamma (s-t)}\|Z_s^n - Z^*_s\|^2\,ds \\
& \leq \tilde{q} \operatorname{\mathbb{E}}\int_t^T e^{\gamma (s-t)} \|z^{n-1}_s - Z^*_s\|^2\,ds,
  \end{aligned}
\end{equation}
where $\gamma > 0$ and $\tilde{q}\in(0,1/2)$ depend only on the Lipschitz constant in Assumption~\ref{assumption:1}. Introducing the following notations
\begin{equation*}
  \begin{aligned}
    a_n&\coloneqq \operatorname{\mathbb{E}}|Y_t^n - Y_t^*|^2 = |v^{\alpha^n}(t,x) - v^*(t,x)|^2,\\
    b_n&\coloneqq \operatorname{\mathbb{E}}\int_t^Te^{\gamma (s-t)}\|Z_s^n - Z_s^*\|^2\,ds,\\
    c_n&\coloneqq \operatorname{\mathbb{E}}\int_t^Te^{\gamma (s-t)}\|z^n_s - Z^n_s\|^2\,ds \leq e^{\gamma(T-t)}\epsilon_n,
  \end{aligned}
\end{equation*}
we further relax the inequality Eq.~\eqref{eq:tmp-801} to (letting $q=2\tilde{q}$)
\begin{equation}
  \label{eq:tmp-841}
  a_n + b_n \leq q(b_{n-1} + c_{n-1}),\quad  \forall n\geq 1.
\end{equation}
Noting that $a_n\geq0$, we substitute $b_n\leq q(b_{n-1}+c_{n-1})$ into the right-hand side of Eq.~\eqref{eq:tmp-841} repeatly:
\begin{equation}
  \label{eq:tmp-858}
  \begin{aligned}
    a_n + b_n &\leq qc_{n-1} + q^2(b_{n-2}+c_{n-2})\\
    &\leq qc_{n-1} + q^2 c_{n-2} + q^3(b_{n-3}+c_{n-3})\\
    &\leq qc_{n-1} + \cdots + q^{n-1}c_1 + q^n(b_0+c_0) \eqqcolon S_n.
  \end{aligned}
\end{equation}
Without loss of generality, we assume $\operatorname{limsup} \epsilon_n < \infty$. Otherwise, the equality to be proved holds trivially. Then, we have $\operatorname{limsup} c_n \leq e^{\gamma(T-t)} \operatorname{limsup}\epsilon_n < \infty$. This means there is a positive integer $M$ such that $c_n$ is bounded by some $c < \infty$ for any $n\geq M$. Hence,
\begin{equation*}
  \begin{aligned}
  S_n &= qc_{n-1} + \cdots + q^{n-M}c_M +\cdots + q^nc_0 + q^nb_0 \\
  &\leq (q + q^2 + \cdots + q^{n-M})c \\
  & \hphantom{\,\leq} + q^{n-M+1}\max\{c_k:0\leq k\leq M-1\} + q^nb_0\\
  &\leq \frac{q}{1-q}c + q^{n-M+1}\max\{c_k:0\leq k\leq M-1\} + q^nb_0.
  \end{aligned}
\end{equation*}
This implies that $S_n$ is also bounded for sufficient large $n$, i.e., $\operatorname{limsup}S_n <\infty$. Observating that $S_n$ satisfies the recurrence equation $S_n = q(S_{n-1}+c_{n-1})$, we can conclude that $\operatorname{limsup}S_n < \frac{q}{1-q} \operatorname{limsup}c_n$ by taking limsup on both sides. Noting $a_n \leq S_n$, we have
\begin{equation*}
  \operatorname*{limsup}_{n\to\infty} a_n  \leq \frac{q}{1-q} \operatorname*{limsup}_{n\to\infty} c_n.
\end{equation*}
Expanding the definitions of $a_n$ and $c_n$ finishes the proof.
\end{proof}
\begin{remark}
  \label{remark:error-offpolicy-proof-proposition-1}
  If $\alpha^b=\alpha^0$ and $\epsilon_n=0$ for any $n$, then $c_n=0$ for any $n$, and Eq.~\eqref{eq:tmp-858} is reduced to $a_n+b_n\leq q^nb_0$. Dropping $b_n$ and expanding the definition of $a_n$ yield the Eq.~\eqref{eq:PDE-convergence}. This shows that our proof can be adopted to prove Proposition~\ref{prop:PDE-convergence}.
\end{remark}

\section{Solving FBSDEs by Optimization}\label{sec:approximation}

In this section, we discuss how to solve the optimization problems encountered in the FBSDE-based subroutines, in which we propose a novel criterion, called the (general) BML criterion. Due to the uncoupling nature of the FBSDEs in our policy iteration algorithms, we focus on solving BSDEs.

\subsection{A practical objective function}\label{sec:BSDE-ML}
The on-policy subroutine involves a BSDE in the form
\begin{equation}
  \label{eq:simple-BSDE}
Y_s = \xi + \int_s^Tf_\tau\,d\tau - \int_s^T\langle Z_\tau,dW_\tau\rangle,\quad \forall s\in[t,T].
\end{equation}
Specifically, for a trial process $z\in \mathbb{H}^2$, we are interested in calculating the distance $ \operatorname{\mathbb{E}}\int_t^T\|Z_s-z_s\|^2\,ds$ between $z$ and the true solution $Z$. The difficulty is that $Z$ is not known and goes into the equation. Hence, we need to find practical objective functions that do not explicitly contain $Z$. For this purpose, the following theorem provides useful insights.

\begin{theorem}
  \label{theorem:BSDE-ML-loss-on-policy}
  Suppose that $\xi\in L^2_{\mathcal{F}_T}$ and $f\in \mathbb{H}^2$. Then,  BSDE~\eqref{eq:simple-BSDE} admits a unique adapted solution $(Y,Z)\in\mathbb{S}^2\times\mathbb{H}^2$. For  adapted process $z\in\mathbb{H}^2$, let $\widetilde{Y}^z_s$ denote the process (not necessarily adapted)
  \begin{equation*}
      \widetilde{Y}^z_s = \xi + \int_s^Tf_\tau\,d\tau - \int_s^T\langle z_\tau, dW_\tau\rangle,\quad \forall s\in[t,T].
  \end{equation*}
  Then, it holds that
  \begin{equation}
    \label{eq:BSDE-ML-loss-on-policy}
      \operatorname{\mathbb{E}}|\widetilde{Y}^z_t - \operatorname{\mathbb{E}} \widetilde{Y}^z_t|^2 = \operatorname{\mathbb{E}}\int_t^T\|Z_s - z_s\|^2\,ds.
  \end{equation}
\end{theorem}
\begin{proof}
  The uniqueness and existence are standard results for BSDEs; see \cite[Chapter~6]{Pham2009} for example. We rewrite the left-hand side of Eq.~\eqref{eq:BSDE-ML-loss-on-policy} as
\begin{equation*}
 \operatorname{\mathbb{E}}|\widetilde{Y}^z_t - Y_t|^2 + 2 \operatorname{\mathbb{E}}[(\widetilde{Y}^z_t - Y_t)(Y_t - \operatorname{\mathbb{E}}\widetilde{Y}^z_t)] + \operatorname{\mathbb{E}}|Y_t - \operatorname{\mathbb{E}}\widetilde{Y}^z_t|^2.
\end{equation*}
Due to the fact that $\mathcal{F}_t$ contains only $\mathbb{P}$-null sets, we know that $Y_t = \operatorname{\mathbb{E}}Y_t$ holds almost surely. Moreover,
  \begin{equation*}
    \operatorname{\mathbb{E}}\widetilde{Y}^z_t = \operatorname{\mathbb{E}}\biggl[\xi + \int_t^Tf_s\,ds\biggr] = \operatorname{\mathbb{E}} Y_t.
  \end{equation*}
Thus, $Y_t- \operatorname{\mathbb{E}}\widetilde{Y}^z_t$ is almost surely zero and 
  \begin{align*}
    \operatorname{\mathbb{E}}|\widetilde{Y}^z_t - \operatorname{\mathbb{E}} \widetilde{Y}^z_t|^2 &= \operatorname{\mathbb{E}}|\widetilde{Y}^z_t - Y_t|^2 \\
&= \operatorname{\mathbb{E}}\biggl[- \int_t^T\langle z_s,dW_s\rangle + \int_t^T\langle Z_s,dW_s\rangle\biggr]^2.
  \end{align*}
  Thus, the desired equality holds due to It\^o's isometry.
\end{proof}
\begin{remark}
  By Remark~\ref{remark:admissibility}, the BSDE in the on-policy subroutine satisfies the conditions here. Thus, $ \operatorname{\mathbb{E}}|\widetilde{Y}^z_t - \operatorname{\mathbb{E}} \widetilde{Y}^z_t|^2$ can be used in the place of the objective function. We call this the special BML criterion, where its general form is discussed in the next subsection.
\end{remark}

An intuitive explanation of the BML criterion is based on the measurability. By definition, $(\widetilde{Y}^z,z)$ has already satisfied the stochastic integral relationship as $(Y,Z)$. Not surprisingly, this is not sufficient to conclude that it is a solution, as $z$ is just arbitrarily selected. The key is that a true pair of solution $(Y,Z)$ should also be adapted. That is to say, $\widetilde{Y}^z_s$ should be $\mathcal{F}_s$-measurable for any $s\in[t,T]$. This is not a trivial matter since the definition of $\widetilde{Y}_s^z$ involves the ``future'' information, particularly the $\{W_\tau\}_{s\leq \tau\leq T}$. Assume, however, that $\widetilde{Y}^z_{t'}$  has been proved to be $\mathcal{F}_{t'}$-measurable. Then it is safe to conclude that $\widetilde{Y}^z_s$ is $\mathcal{F}_s$-measurable for any $s\in[t',T]$. This is because for any $s\in [t',T]$, we have
\begin{align*}
  \widetilde{Y}^z_{t'} &= \xi + \int_{t'}^Tf_\tau\,d\tau - \int_{t'}^T\langle z_\tau,dW_\tau\rangle \\
& = \widetilde{Y}^z_{s} + \int_{t'}^sf_\tau\,d\tau - \int_{t'}^s\langle z_\tau,dW_\tau\rangle.
\end{align*}
Clearly, the integral part is $\mathcal{F}_s$-measurable. As a result, $\widetilde{Y}^z_{s}$ is $\mathcal{F}_{s}$-measurable because $\widetilde{Y}^z_{t'}$ is $\mathcal{F}_{s}$-measurable (recall that $\mathcal{F}_{t'}\subset\mathcal{F}_s$ if $t'\leq s$).

The left-hand side of Eq.~\eqref{eq:BSDE-ML-loss-on-policy} serves as a criterion of the measurability loss of $\widetilde{Y}^z_t$ with respect to $\mathcal{F}_t$. Recall that $\mathcal{F}_t = \sigma(\mathcal{N}\cup\sigma(W_t))$, where $\mathcal{N}$ is the collection of $\mathbb{P}$-null sets and $\sigma(W_t)$ is the trivial $\sigma$-algebra with $W_t=0$. $\widetilde{Y}^z_t$ is $\mathcal{F}_t$-measurable if and only if $\widetilde{Y}_t^z$ is a constant almost surely. To put it in another way, $\widetilde{Y}^z_t$ should be equal to the expectation almost surely. This is exactly the case that Eq.~\eqref{eq:BSDE-ML-loss-on-policy} equals 0. 

\subsection{The BML Criterion}
According to Theorem~\ref{theorem:BSDE-ML-loss-on-policy}, the distance $ \operatorname{\mathbb{E}}\int_t^T\|Z_s-z_s\|^2$ can be calculated with only samples of $\xi,f$ and $W$ in BSDE~\eqref{eq:simple-BSDE}. This allows an optimization-based approach to solving the $Z$ part of solutions by parameterizing the trial process $z$, and then minimizing the practical objective function. However, in many applications, obtaining the $Y$ part of solutions may be appealing as well. Indeed, according to Feynman-Kac's formula, $Y_t$ is the value function at $(t,x)$. If we manage to find the exact or an approximated solution of $Y$, then we also find a method to solve PDEs in the form of Eq.~\eqref{eq:PDE-characterization}.

In the proof of Theorem~\ref{theorem:BSDE-ML-loss-on-policy}, we ultilize the fact that $ \operatorname{\mathbb{E}}\widetilde{Y}^z_t = \operatorname{\mathbb{E}}Y_t = Y_t$ holds almost surely. Unfortunately, $\widetilde{Y}^z$ is not a suitable replacement for $Y$ in applications. The major issue is that the definition of $\widetilde{Y}^z$ is ``anticipated''. Even if $z\equiv Z$, calculating the value of $\widetilde{Y}^z_t$ by its definition requires samples of $\{W_s;t\leq s\leq T\}$ and $\{f_s;t\leq s\leq T\}$, which are not available at the time instant $t$. Nevertheless, $\widetilde{Y}^z$ differs from the true solution only by a martingale term, and this difference can be eliminated by taking conditional expectation:
\begin{equation}
  \label{eq:conditional-equality}
\operatorname{\mathbb{E}}[\widetilde{Y}^z_s\,|\,\mathcal{F}_s] = 
\operatorname{\mathbb{E}}[Y_s\,|\,\mathcal{F}_s] = Y_s,\quad\mathbb{P}\text{-a.s.},\quad\forall s\in[t,T].
\end{equation}
In light of this, we extend  Theorem~\ref{theorem:BSDE-ML-loss-on-policy} by adding the distance between a trial solution $\tilde{v}\in\mathbb{S}^2$ and the true solution $Y$.

\begin{theorem}
  \label{theorem:BSDE-ML-loss-on-policy-extended}
  Let the condition of Theorem~\ref{theorem:BSDE-ML-loss-on-policy} hold and use the same notation. Then, for any adapted process $\tilde{v}\in\mathbb{S}^2$, there is
\begin{equation}
  \label{eq:general-BSDE-ML}
  \begin{aligned}
\operatorname{\mathbb{E}}\int_t^T|\widetilde{Y}^z_s - \tilde{v}_s|^2\,\nu(ds) &= \operatorname{\mathbb{E}}\int_t^T\int_s^T\|Z_\tau-z_\tau\|^2\,d\tau\,\nu(ds) \\
&\hphantom{\,=} + \operatorname{\mathbb{E}}\int_t^T|Y_s - \tilde{v}_s|^2\,\nu(ds),
  \end{aligned}
\end{equation}
where $\nu$ is an arbitrary $\sigma$-finite measure on $[t,T]$.
\end{theorem}
\begin{proof}
  Similarly, we prove Eq.~\eqref{eq:general-BSDE-ML} by splitting the square term into three terms and showing that the expectation of the cross term is zero. As $\nu$ is $\sigma$-finite, we are able to change the order of expectation and integration, and thus, the left-hand side of Eq.~\eqref{eq:general-BSDE-ML} equals
\begin{equation*}
\int_t^T\biggl[ \operatorname{\mathbb{E}}|\widetilde{Y}^z_s - Y_s|^2 + 2 \operatorname{\mathbb{E}}[(\widetilde{Y}^z_s - Y_s)(Y_s - \tilde{v}_s)] + \operatorname{\mathbb{E}}|Y_s - \tilde{v}_s|^2\biggr]\,\nu(ds).
\end{equation*}
The first term can be transformed with It\^o's isometry:
\begin{align*}
\operatorname{\mathbb{E}}\int_t^T|\widetilde{Y}^z_s - Y_s|^2\,\nu(ds) 
&= \operatorname{\mathbb{E}}\int_t^T\biggl|\int_s^T\langle z_\tau-Z_\tau,dW_\tau\rangle\biggr|^2\,\nu(ds) \\
&= \operatorname{\mathbb{E}}\int_t^T\int_s^T\| z_\tau-Z_\tau\|^2\,d\tau\,\nu(ds).
\end{align*}
The second term vanishes according to the tower property of conditional expectation:
\begin{equation*}
\begin{aligned}
\operatorname{\mathbb{E}}[(\widetilde{Y}^z_s - Y_s)(Y_s - \tilde{v}_s)] &= \operatorname{\mathbb{E}}\bigl[\operatorname{\mathbb{E}}[(\widetilde{Y}^z_s - Y_s)(Y_s - \tilde{v}_s)\,|\,\mathcal{F}_s]\bigr] \\
&= \operatorname{\mathbb{E}}\bigl[(Y_s - \tilde{v}_s)\operatorname{\mathbb{E}}[(\widetilde{Y}^z_s - Y_s)\,|\,\mathcal{F}_s]\bigr] \\
&= 0.
\end{aligned}
\end{equation*}
The last equality comes from the fact that $\mathbb{E}[(\widetilde{Y}^z_t - Y_t)\,|\,\mathcal{F}_s]$ is zero almost surely.
\end{proof}
\begin{remark}
We call Eq.~\eqref{eq:general-BSDE-ML} the general BML criterion.
While the special BML criterion focuses solely on the $Z$ part, its generalization takes the $Y$ part into account as well. We do this by relacing $ \operatorname{\mathbb{E}}Y^z_t$ with $\tilde{v}_s$. Moreover, Eq.~\eqref{eq:general-BSDE-ML} introduces a measure on the time space $[t,T]$. The left-hand side of Eq.~\eqref{eq:general-BSDE-ML} actually descrbies the distance between $\widetilde{Y}^z$ and $\tilde{v}$ on the product space $(\Omega\times[t,T], \mathbb{P}\otimes\nu)$. On the other hand, this practical objective function can also be interpreted as the distance between $(\tilde{v},z)$ and $(Y,Z)$ using this product measure. Under this generalization, we are given the freedom of choosing $\nu$ when comapring the trial solution with the true solution. In particular, if $\nu$ is set to the Dirac measure centered on $t$ and $\tilde{v}$ to $ \operatorname{\mathbb{E}}[\widetilde{Y}_s^z\,|\,\mathcal{F}_s]$, then it comes to the special BML criterion. It is also possible to choose different settings of $\nu$ and $(\tilde{v},z)$. It will be discussed shortly and how the general BML criterion degenerates into existing methods.
\end{remark}
\begin{remark}
\label{remark:decoupled-z-tildev}
It is worth noting that if the choice of $\tilde{v}$ does not rely on $z$, then the two terms in Eq.~\eqref{eq:general-BSDE-ML} are decoupled. This means that the gradient with respect to $\tilde{v}$ is independent of the gradient with respect to $z$. Therefore, $z$ and $\tilde{v}$ can be optimized independently. In this case, our estimation of $Z$ does not affect the estimation of $Y$, and vice versa. One advantage of this property is that even if $z$ is actually far from the true solution $Z$, it is still possible to have a good estimation of $Y$ that is fairly accuracy. As an application, we could fix $z\equiv0$ and focus solely on estimating of $Y$ by optimizing only $\tilde{v}$. According to our analysis, this simply results in the distance between $z$ and $Z$ remaining constant, and we may still be able to obtain a reasonable estimation of $Y$ if the general BML criterion reaches its minimum. 
\end{remark}

By choosing $\nu=\delta_t$ and $\tilde{v}(s,\omega)\equiv y_0$, we recover the popular Deep BSDE method proposed in \cite{Han2018}. There, $\delta_t$ is the Dirac measure centered at $t$ and $y_0\in\mathbb{R}$ does not change along with time $s$ and the sample event $\omega$. The general BML criterion is then reduced to $\operatorname{\mathbb{E}}|\widetilde{Y}_t - y_0|^2$, which can be interpreted as
\begin{equation}
  \label{eq:tmp-654}
\begin{aligned}
 &\hphantom{\,=\,} \operatorname{\mathbb{E}}\biggl|\xi - \biggl(y_0 - \int_t^Tf_s\,ds + \int_t^T\langle z_s,dW_s\rangle\biggr)\biggr|^2\\
&= \operatorname{\mathbb{E}}\int_t^T\|z_s - Z_s\|^2\,ds + \operatorname{\mathbb{E}}|Y_t - y_0|^2
\end{aligned}
\end{equation}
by Theorem~\ref{theorem:BSDE-ML-loss-on-policy-extended}. The original motivation of Deep BSDE method is to examine the process
\begin{equation*}
\widetilde{Y}^{z,y_0}_s= y_0 -\int_t^sf_\tau\,d_\tau, + \int_t^s\langle z_\tau,dW_\tau\rangle.
\end{equation*}
In fact, this is a forward stochastic differential equation. One can relate it to BSDE~\eqref{eq:simple-BSDE} by requiring $Y_T^{z,y_0}=\xi$ holds almost surely, i.e., forcing $\operatorname{\mathbb{E}}|\xi - Y_T^{z,y_0}|^2=0$. This is exactly the criterion used in Deep BSDE method. If the choices of $y_0$ and $z$ do not depend on each other, Remark~\ref{remark:decoupled-z-tildev} reveals that this criterion is equivalent to $ \operatorname{\mathbb{E}}|Y_t-y_0|^2$ when one is only interested in estimating the value of $Y_t$. We should also mention that Deep BSDE method applies for a wider class of BSDEs other than the simple form Eq.~\eqref{eq:simple-BSDE}. There, the generator $f_s$ is coupled with $(Y_s, Z_s)$ by a nonlinear function $f$. In that case, Eq.~\eqref{eq:BSDE-ML-loss-on-policy} and Eq.~\eqref{eq:general-BSDE-ML} are no longer valid. We will briefly discuss that topic at the end of this section.

By choosing $\nu(ds)=ds$ and $z\equiv0$, we recover the martingale approach proposed in \cite{jia_zhou_2022}. The general BML criterion is then reduced to
\begin{align*}
&\hphantom{\,=\,} \operatorname{\mathbb{E}}\int_t^T\biggl|\biggl(\xi + \int_t^Tf_\tau\,d\tau\biggr) - \biggl( \tilde{v}_s + \int_t^sf_\tau\,d\tau\biggr)\biggr|^2\,ds\\
&= \operatorname{\mathbb{E}}\int_t^T\int_s^T\|Z_\tau\|^2\,d\tau\,ds + \operatorname{\mathbb{E}}\int_t^T|Y_s-\tilde{v}_s|^2\,ds
\end{align*}
by Theorem~\ref{theorem:BSDE-ML-loss-on-policy-extended}. In the martingale approach, one takes no care of the $Z$ part of solution and just set the trial solution $z$ to zero. This treatment is permitted by Remark~\ref{remark:decoupled-z-tildev} as well. Minimizing the distance between $\widetilde{Y}^z$ and $\tilde{v}$ with $z\equiv0$ is indeed equivalent to minimizing the distance between $\tilde{v}$ and the true solution $Y$. The similar result is reported along with the martingale approach in \cite{jia_zhou_2022}, but there is no discussion about its connection to Deep BSDE method.

\begin{corollary}
  Let the condition of Theorem~\ref{theorem:BSDE-ML-loss-on-policy} hold and use the same notation. For any $y_0\in\mathbb{R}$ and $z\in\mathbb{H}^2$, let $\widehat{Y}_s^{z,y_0}$ denote the process 
  \begin{equation*}
    \widehat{Y}_s^{z,y_0} = y_0 - \int_t^sf_\tau\,d\tau + \int_t^s\langle z_\tau,dW_\tau\rangle,\quad \forall s\in[t,T].
  \end{equation*}
  Then, it holds that
  \begin{equation*}
    \min_{y_0\in\mathbb{R}} \operatorname{\mathbb{E}} |\widehat{Y}_T^{z,y_0}-\xi|^2 = \operatorname{\mathbb{E}} |\widetilde{Y}^z_0 - \operatorname{\mathbb{E}}\widetilde{Y}^z_0|^2.
  \end{equation*}
\end{corollary}
\begin{proof}
  This is a direct consequence of Theorem~\ref{theorem:BSDE-ML-loss-on-policy} and Eq.~\eqref{eq:tmp-654}.
\end{proof}
\begin{remark}
  In general, the criterion $ \operatorname{\mathbb{E}}|\widehat{Y}_s^{z,y_0}-\xi|^2$, used in Deep BSDE method, depends on both $z$ and $y_0$. If $y_0$ is optimized with fixed $z$, it comes to the special BML criterion.
\end{remark}

\subsection{Optimize with the proposed criterion}
In this subsection, we illustrate how to solve a BSDE by optimizing the proposed criterion. As discussed at the end of the last subsection, the general BML criterion is a class of objective functions and choosing different $(\nu,\tilde{v},z)$ leads to different specific objective functions. We summarize four sets of $(v,\tilde{v},z)$ in Table~\ref{table:four-criteria} and refer to them as Set (a)--(d). It should be noted that Set (a) and Set (c) are used in Deep BSDE method and the martingale approach, respectively. Set (b) corresponds to the special BML criterion proposed in Theorem~\ref{theorem:BSDE-ML-loss-on-policy}, while Set (d) is considered here to show the general form helps us finding new objective functions. It should be pointed out that $\tilde{v}_s=\operatorname{\mathbb{E}}[\widetilde{Y}^z_s\,|\,\mathcal{F}_s]$in Set (b) is merely provided for completeness, and is not required for calculations. We stress that these four sets cover only a small part of the general BML criterion, and it is always possible to design appropriate forms of $\nu,\tilde{v}$ and $z$ based on specific requirements. In order to focus on ideas, we test these four criteria on the following toy example. A more involved example will be discussed in the next section.

\begin{table}
  \caption{\label{table:four-criteria}Four special cases of the general BML criterion.}
  \centering
  \medskip
  \begin{tabular}{ccccc}
    \toprule
    Name & $d\nu/ds$ & $\tilde{v}$ & $z$ & Practical objective function \\
    \midrule
    Set (a) & $\delta_t$ & $y_0$ & $z_s$ & $ \operatorname{\mathbb{E}}|\widetilde{Y}^z_0-y_0|^2$ \\
    Set (b) & $\delta_t$ & $ \operatorname{\mathbb{E}}[\widetilde{Y}^z_s\,|\,\mathcal{F}_s]$ & $z_s$ & $\operatorname{\mathbb{E}}|\widetilde{Y}^z_0 - \operatorname{\mathbb{E}}\widetilde{Y}^z_0|^2$ \\
    Set (c) & $1$ & $\tilde{v}_s$ & 0 & $\operatorname{\mathbb{E}}\int_t^T|\widetilde{Y}^0_s - \tilde{v}_s|^2\,ds$ \\
    Set (d) & $1$ & $\tilde{v}_s$ & $z_s$ & $ \operatorname{\mathbb{E}}\int_t^T|\widetilde{Y}^z_s - \tilde{v}_s|^2\,ds$\\
    \bottomrule
  \end{tabular}
\end{table}

\begin{example}
  \label{example:1}
  Solve the BSDE~\eqref{eq:simple-BSDE} with $t=0, T=1, f(s,\omega)\equiv-1, \xi=\langle W_T,W_T\rangle/n$, where $n$ is the dimension of the Brownian motion and is set to 100.
\end{example}

We parameterize the trial processes in Table~\ref{table:four-criteria} as $\tilde{v}_s=W_s^\intercal \theta_y W_s,~z_s=2\theta_z W_s$. Additionally, Set (a) involves optimizing a standalone variable $y_0$. The Brownian motion is simulated with time step $\Delta t=0.01$. The expectation is estimated via Monte Carlo simulation with sample size $M=16$. Integration is approximated with the Euler method. Optimization method is chosen as the standard stochastic gradient descent (SGD) method with different learning rates: $1.0\times10^{-1}$ for $y_0$, $1.0\times 10^{-3}$ for $\theta_z$ and $1.0\times10^{-5}$ for $\theta_y$. The initial values of $y_0,\theta_y,\theta_z$ are set to $1.0, -1.0, -1.0$, respectively. For each set, we perform 200 gradient steps and repeat the whole procedure 10 times with different random seeds. The true value of these variables are obtained via  theoretical analysis. It can be verified by It\^o's formula that $Y_s=\langle W_s,W_s\rangle/n, Z_s=2W_s/n$ is a pair of adapted solution. This solution is also unique because $\xi\in L^2_{\mathcal{F}_T}$ and $f\in\mathbb{H}^2$. Thus, the optimal values are $\theta_y^*=\theta_z^*=1/n$. Additionally, the $y_0$ in Set (a) is used to estimate the value of $Y_0$, and thus, has the optimal value $y_0^*=0$. Results are reported in Figure~\ref{fig:1}.

\begin{figure}
  \centering
  \includegraphics[width=.4\textwidth]{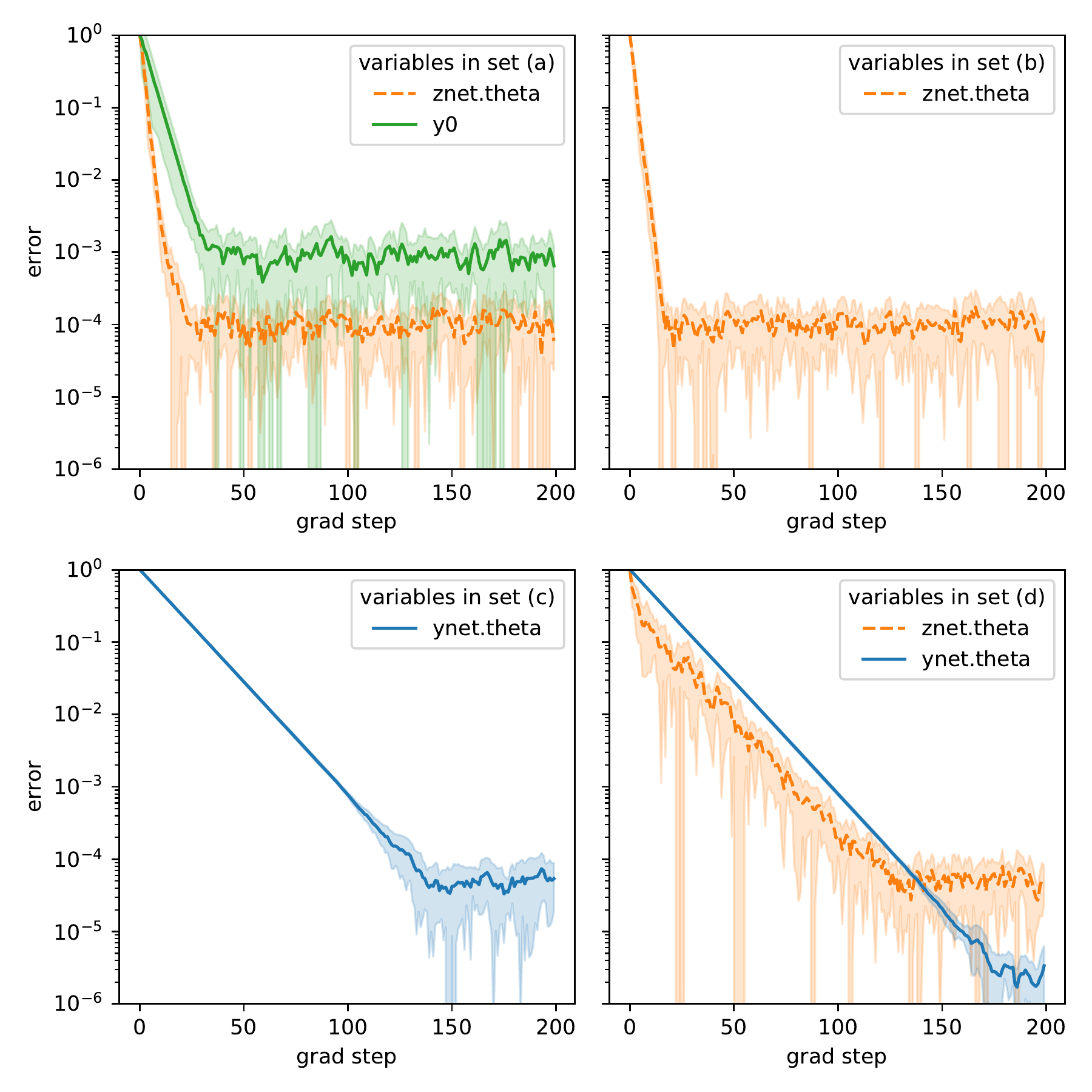}
  \caption{\label{fig:1}The absolute errors of $\theta_y,\theta_z,y_0$ at each gradient steps for Example~\ref{example:1}. From left to right and from top to bottom, the subplots correspond to Set (a), (b), (c) and (d). The solid lines and shaded areas indicate the mean and standard deviation of absolute errors for 10 runs.}
\end{figure}

Figure~\ref{fig:1} plots the absolute errors of $\theta_y,\theta_z,y_0$ at each gradient steps in four subplots, corresponding to the four sets in Table~\ref{table:four-criteria}. It can be seen that all variables in these sets converge to their true values with fairly high accuracy in 200 gradient steps. There are two interesting phenomena of convergence trends. The first one is that $\theta_z$ converges very quickly in Set (a) and Set (b) with almost the same rate, but is slightly slower in Set (d). The second one is that the $\theta_y$ in Set (d) converges to a better value than that in Set (c). We can explain them with the help Theorem~\ref{theorem:BSDE-ML-loss-on-policy-extended}. 

According to Eq.~\eqref{eq:general-BSDE-ML}, the objective functions in these four sets can be interpreted as follows. Set (a) minimizes $ \operatorname{\mathbb{E}}\int_t^T\|z_s-Z_s\|^2\,ds$, which also is the term to be minimized in Set (b) plus an additional term $ \operatorname{\mathbb{E}}|y_0-Y_0|^2$. Therefore, the gradients for $\theta_z$ computed in Set (a) and Set (b) should be identical except for the noise introduced by Monte Carlo sampling. This is the reason why the convergent behavior of $\theta_z$ is similar in these two sets. On the other hand, the $\theta_z$ in Set (d) appears in a double integration $ \operatorname{\mathbb{E}}\int_t^T\int_s^T\|z_\tau-Z_\tau\|^2\,d\tau\,ds$ due to the choice of $\nu$. In order to explain the second phenomena, we need to review the proof of Theorem~\ref{theorem:BSDE-ML-loss-on-policy-extended}. There, the cross-term is eliminated by taking expectation. However, in practice, this term does not vanish if we use Monte Carlo estimation. A simple analysis shows that its variance is proportional to $|Y_s-\tilde{v}_s|^2$, which is also minimized in Set (d), but not in Set (c). Therefore, a slight performance improvement in Set (d) compared to Set (c) is expected.

In addition, Theorem~\ref{theorem:BSDE-ML-loss-on-policy-extended} gives us the hint of choosing better learning rates. Take $y_0$ as an example at first. In Set (a), $y_0$ appears in the term $ \operatorname{\mathbb{E}}|y_0-Y_0|^2$. In optimization theory, the optimal learning rate for quadratic function $a\|x-x^*\|^2$ is $\frac{1}{2a}$; see, for example, Nesterov \textit{et al}~\cite{Nesterov2018}. Thus, the optimal learning rate for $y_0$ is 0.5. Considering the noise effect, we select a much smaller and thus safer value 0.1. For $\theta_y$, the analysis becomes a little more complicated. Eq.~\eqref{eq:general-BSDE-ML} tells us that $\theta_y$ appears in the term $ \operatorname{\mathbb{E}}\int_t^T|Y_s - \tilde{v}_s|^2\,ds$. Substituting $Y_s=\theta^*_y\|W_s\|^2$ and $\tilde{v}_s=\theta_y\|W_s\|^2$ into it yields $(\theta_y-\theta_y^*)\int_t^T \operatorname{\mathbb{E}}\|W_s\|^4\,ds$. By integrating on a sphere, we can calculate that $\int_t^T \operatorname{\mathbb{E}}\|W_s\|^4\,ds=n(n+2)(T-t)^3/3$. Thus, the optimal learning rate for $\theta_y$ is in the order of $10^{-4}$. Based on this, we select the value $1\times10^{-5}$.

\subsection{Other type of BSDEs}
The BSDE~\eqref{eq:simple-BSDE} considered before is only a basic type of general BSDEs. In many applications, for example in our off-policy subroutine, the generator $f$ may be unknown and is expressed as $f(s,Y_s,Z_s)$ with a deterministic (or even random) coefficient $f(\cdot,\cdot,\cdot)$. Elementary extensions of Theorem~\ref{theorem:BSDE-ML-loss-on-policy} and Theorem~\ref{theorem:BSDE-ML-loss-on-policy-extended} in this line are provided below.

Consider the following BSDE
\begin{equation}
  \label{eq:BSDE-couple-Z}
  Y_s = \xi + \int_s^Tf(\tau,Z_\tau)\,d\tau - \int_s^T\langle Z_\tau,dW_\tau\rangle,~~ \forall s\in[t,T],
\end{equation}
where the generator $f$ is only coupled with $Z$. For any $z\in\mathbb{H}^2$, we can still introduce the process $\widetilde{Y}^z$ by replacing $Z$ with $z$. Sadly, Eq.~\eqref{eq:conditional-equality} fails to hold because $f(s,z_s)$ may be not equal to $f(s,Z_s)$. As a result, Theorem~\ref{theorem:BSDE-ML-loss-on-policy} and Theorem~\ref{theorem:BSDE-ML-loss-on-policy-extended} no longer hold. Nevertheless, The general BML criterion for trial solutions $(\tilde{v},z)$ can still be calculated and optimized, and obviously the true $(Y,Z)$ is a global minimum of this criterion. Thus, the proposed criterion equals zero is a necessary condition for solving such a BSDE. Moreover, we are able to say it is also  a sufficient condition to some extent. 

\begin{proposition}
\label{proposition:nonlinear-BSDE-ML-loss}
Suppose that $\xi\in L^2_{\mathcal{F}_T}$ and $f:\Omega\times[t,T]\times\mathbb{R}^d\to\mathbb{R}$ satisfies the following conditions: 1) for any $z\in\mathbb{R}^d$, $f(s,z)$ is adapted; 2) $f(s,0)\in\mathbb{H}^2$; 3) there exists a constant $L$ such that for any $z^1,z^2\in\mathbb{R}^d$, $$|f(s,z^1)-f(s,z^2)|\leq L|z^1-z^2|, ~~ds\otimes d\mathbb{P}\textrm{-a.e.}$$ on $[t,T]\times\Omega$. Then, BSDE~\eqref{eq:BSDE-couple-Z} admits a unique adapted solution $(Y,Z)\in\mathbb{S}^2\times\mathbb{H}^2$. For any adapted process $z\in\mathbb{H}^2$, let $\widetilde{Y}^z_s$ denote the process (not necessarily adapted)
  \begin{equation*}
      \widetilde{Y}^z_s = \xi + \int_s^Tf(\tau,z_\tau)\,d\tau - \int_s^T\langle z_\tau, dW_\tau\rangle,\quad\forall s\in[t,T].
  \end{equation*}
Then, $\operatorname{\mathbb{E}}\int_t^T\|Z_s-z_s\|^2\,ds$ equals zero if and only if $ \operatorname{\mathbb{E}}|\widetilde{Y}^z_t - \operatorname{\mathbb{E}}\widetilde{Y}^z_t|^2$ equals zero.
\end{proposition}
\begin{proof}
  The uniqueness and existence are standard results for BSDEs; see \cite[Chapter~6]{Pham2009} for example. 

Let  $\operatorname{\mathbb{E}}\int_t^T\|Z_s-z_s\|^2\,ds=0$ be true. Noting the assumptions on $f$, for any $s\in[t,T]$, we have
  \begin{align*}
 &\hphantom{\,=\,}   \operatorname{\mathbb{E}}\biggl[\int_s^Tf(\tau,Z_\tau)\,d\tau - \int_s^Tf(\tau,z_\tau)\,d\tau\biggr]^2 \\
 &\leq(T-s) \operatorname{\mathbb{E}}\int_s^T|f(\tau,Z_\tau) - f(\tau,z_\tau)|^2\,d\tau\\
& \leq L^2(T-s) \operatorname{\mathbb{E}}\int_s^T\|Z_\tau-z_\tau\|^2\,d\tau= 0.
  \end{align*}
Furthermore, according to It\^o's isometry, there is
\begin{equation*}
  \operatorname{\mathbb{E}}\biggl[\int_s^T\langle Z_\tau,dW_\tau\rangle - \int_s^T\langle z_\tau,dW_\tau\rangle\biggr]^2 = \operatorname{\mathbb{E}}\int_s^T\|Z_\tau-z_\tau\|^2=0.
\end{equation*}
Hence, $\widetilde{Y}^z_s=Y_s$ holds almost surely for any $s\in[t,T]$. In particular, $\operatorname{\mathbb{E}}|\widetilde{Y}^z_t - \operatorname{\mathbb{E}}\widetilde{Y}^z_t|^2= \operatorname{\mathbb{E}}|Y_t- \operatorname{\mathbb{E}}Y_t|^2=0$. This proves the ``only if'' part.

In order to prove the ``if'' part, we consider the BSDE
\begin{equation}
  \label{eq:tmp-771}
  \widehat{Y}_s = \xi + \int_s^T\hat{f}_\tau\,d\tau - \int_s^T\langle \widehat{Z}_\tau,dW_\tau\rangle,\quad\forall s\in[t,T],
\end{equation}
where $\hat{f}_\tau\coloneqq f(\tau,z_\tau)$. This is the type of BSDE studied in previous subsections. By assumptions on $f$, the process $\hat{f}\in\mathbb{H}^2$. Applying Theorem~\ref{theorem:BSDE-ML-loss-on-policy} to BSDE~\eqref{eq:tmp-771} concludes that the solution $(\widehat{Y},\widehat{Z})\in\mathbb{S}^2\times\mathbb{H}^2$ uniquely exists and $$ \operatorname{\mathbb{E}}\int_t^T\|\widehat{Z}_s-z_s\|^2\,ds = \operatorname{\mathbb{E}}|\widetilde{Y}^z_t - \operatorname{\mathbb{E}}\widetilde{Y}^z_t|^2=0.$$ Theorefore, $z_s=\widehat{Z}_s$ holds $ds\otimes d\mathbb{P}$ almost everywhere.

In view of BSDE~\eqref{eq:BSDE-couple-Z} and BSDE~\eqref{eq:tmp-771}, we denote $$\overline{Y}\coloneqq Y-\widehat{Y},~~ \overline{Z}\coloneqq Z - \widehat{Z}, ~~\bar{f}_s\coloneqq f(s,Z_s)-\hat{f}_s.$$ Let $\gamma$ be a positive constant such that $\gamma > 2L^2$. By applying It\^o's formula to $e^{\gamma s}|\overline{Y}_s|^2$, we obtain
\begin{equation}
\label{eq:tmp-776}
\begin{aligned}
&\hphantom{\,=\,}\operatorname{\mathbb{E}}e^{\gamma t}|\overline{Y}_t|^2 + \operatorname{\mathbb{E}}\int_t^Te^{\gamma s}(\gamma|\overline{Y}_s|^2 + \|\overline{Z}_s\|^2)\,ds \\
& = 2 \operatorname{\mathbb{E}}\int_t^Te^{\gamma s}\overline{Y}_s\bar{f}_s\,ds - 2 \operatorname{\mathbb{E}}\int_t^Te^{\gamma s}\overline{Y}_s\langle \overline{Z}_s,dW_s\rangle.
\end{aligned}
\end{equation}
A standard analysis based on Burkholder-Davis-Gundy inequality shows that  the second term vanishs; see the proof of \cite[ Theorem~6.2.1]{Pham2009}. On the other hand, for any $s\in[t,T]$,
\begin{equation}
  \label{eq:tmp781}
2\overline{Y}_s\bar{f}_s\leq \gamma|\overline{Y}_s|^2 + \frac{1}{\gamma}|\bar{f}_s|^2 \leq \gamma|\overline{Y}_s|^2 + \frac{L^2}{\gamma}\|Z_s - z_s\|^2.
\end{equation}
Noting $L^2/\gamma < 1/2$, Eq.~\eqref{eq:tmp-776} and Eq.~\eqref{eq:tmp781}, there is
\begin{align*}
\operatorname{\mathbb{E}}\int_t^Te^{\gamma s}\|Z_s-\widehat{Z}_s\|^2\,ds &\leq \frac{1}{2} \operatorname{\mathbb{E}}\int_t^Te^{\gamma s}\|Z_s - z_s\|^2\,ds \\
& = \frac{1}{2} \operatorname{\mathbb{E}}\int_t^Te^{\gamma s}\|Z_s - \widehat{Z}_s\|^2\,ds.
\end{align*}
The last equality comes from the fact that $z_s=\widehat{Z}_s$ holds $ds\otimes d\mathbb{P}$ almost everywhere. Hence, $ \operatorname{\mathbb{E}}\int_t^Te^{\gamma s}\|Z_s-\widehat{Z}_s\|^2\,ds = 0$. Replacing $\widehat{Z}_s$ with $z_s$ again finishes our proof. 
\end{proof}
\begin{remark}
  When the linear BSDE~\eqref{eq:simple-BSDE} generalizes to the nonlinear BSDE~\eqref{eq:BSDE-couple-Z}, the equality $ \operatorname{\mathbb{E}}|\widetilde{Y}_t^z- \operatorname{\mathbb{E}}\widetilde{Y}_t^z|^2= \operatorname{\mathbb{E}}\int_t^T\|Z_s-z_s\|^2$ may be not true. But, we can still say that the left-hand side is zero if and only if the right-hand side is zero. Thus, we believe that minimizing the special BML criterion is still a reasonable choice for solving the $Z$ part of BSDEs.
\end{remark}
\begin{remark}
  Under Assumptions~\ref{assumption:1}, the BSDE in the off-policy subroutine satisfies the conditions here.
\end{remark}

\begin{proposition}
  \label{proposition:nonlinear-BSDE-ML-loss-extended}
  Let the condition of Proposition~\ref{proposition:nonlinear-BSDE-ML-loss} hold and use the same notation. Let $\tilde{v}_s$ be an adapted process in $\mathbb{S}^2$ and $\nu$ be a $\sigma$-finite measure on $[t,T]$. Then,
\begin{equation*}
\operatorname{\mathbb{E}}\int_s^T\|Z_\tau-z_\tau\|^2\,d\tau = \operatorname{\mathbb{E}}|Y_s - \tilde{v}_s|^2 = 0,\quad \nu\textrm{-a.e.},\quad\forall s\in[t,T]
\end{equation*}
if and only if $ \operatorname{\mathbb{E}}\int_t^T|\widetilde{Y}^z_s - \tilde{v}_s|^2\,\nu(ds)=0$.
\end{proposition}
\begin{proof}
  The sketch of this proof is similar to that of Proposition~\ref{proposition:nonlinear-BSDE-ML-loss} except for a few minor differences concerning the additional $\tilde{v}$ and $\nu$. A brief description of it is provided below, and readers may refer to Proposition~\ref{proposition:nonlinear-BSDE-ML-loss}'s proof for more explanations.

  We prove the ``only if'' part at first. By the assumption on $f$, we are able to show that $ \widetilde{Y}^z_s=Y_s$ holds $d\nu\times d\mathbb{P}$-a.e.. Hence,
\begin{equation*}
\operatorname{\mathbb{E}}\int_t^T|\widetilde{Y}^z_s - \tilde{v}_s|^2\,\nu(ds) \leq 2 \operatorname{\mathbb{E}}\int_t^T|\widetilde{Y}^z_s -Y_s|^2 + |Y_s - \tilde{v}_s|^2\,\nu(ds),
\end{equation*}
which equals zero by assumptions.

Then we prove the ``if'' part. Consider BSDE~\eqref{eq:tmp-771} with $\hat{f}_\tau\coloneqq f(\tau,z_\tau)$. Applying Theorem~\ref{theorem:BSDE-ML-loss-on-policy-extended} to that BSDE, we conclude that the solution $(\widehat{Y},\widehat{Z})\in\mathbb{S}^2\times\mathbb{H}^2$ uniquely exists and that for any $s\in[t,T]$,
\begin{equation}
\operatorname{\mathbb{E}}\int_s^T\|\widehat{Z}_\tau-z_\tau\|^2\,d\tau = \operatorname{\mathbb{E}}|\widehat{Y}_s - \tilde{v}_s|^2 = 0,\quad\nu\text{-a.e.}.
\label{eq:tmp-815}
\end{equation}
Moreover, in view of BSDE~\eqref{eq:BSDE-couple-Z} and BSDE~\eqref{eq:tmp-771}, we have
\begin{align*}
&\hphantom{\,=\,}\operatorname{\mathbb{E}}e^{4L^2s}|Y_s - \widehat{Y}_s|^2 + \operatorname{\mathbb{E}}\int_s^Te^{4L^2\tau}\|Z_\tau - \widehat{Z}_\tau\|^2\,d\tau \\
&\leq \frac{1}{4} \operatorname{\mathbb{E}}\int_s^Te^{4L^2\tau}\|Z_\tau-z_\tau\|^2\,d\tau.
\end{align*}
Integrating on $([t,T],\nu)$ and noting Eq.~\eqref{eq:tmp-815} yield
\begin{align*}
& \hphantom{\,=\,} \operatorname{\mathbb{E}}\int_t^T\int_s^Te^{4L^2\tau}\|Z_\tau - \widehat{Z}_\tau\|^2\,d\tau\,\nu(ds) \\
&\leq \frac{1}{2} \operatorname{\mathbb{E}}\int_t^T\int_s^Te^{4L^2\tau}\|Z_\tau-\widehat{Z}_\tau\|^2\,d\tau\,\nu(ds).
\end{align*}
Hence, for any $s\in[t,T]$,
\begin{equation*}
\operatorname{\mathbb{E}}\int_s^Te^{4L^2\tau}\|\widehat{Z}_\tau-Z_\tau\|^2\,d\tau = \operatorname{\mathbb{E}}e^{4L^2s}|\widehat{Y}_s - Y_s|^2 = 0,\quad\nu\text{-a.e.}.
\end{equation*}
Using Eq.~\eqref{eq:tmp-815} again finishes our proof.
\end{proof}
\begin{remark}
  This extends Theorem~\ref{theorem:BSDE-ML-loss-on-policy-extended} as  Proposition~\ref{proposition:nonlinear-BSDE-ML-loss} extends  Theorem~\ref{theorem:BSDE-ML-loss-on-policy}. For nonlinear BSDEs, the (general) BML criterion reaches zero can be interpreted as a necessary and sufficient condition of finding solutions. 
\end{remark}

The BSDE encountered in the off-policy subroutine is a special case of the BSDE considered in this subsection, where the generator $f(s,Z)$ is linear to $Z$. While Proposition~\ref{proposition:nonlinear-BSDE-ML-loss} and Proposition~\ref{proposition:nonlinear-BSDE-ML-loss-extended} provide general treatments for nonlinear generator, a generator linearly coupled in $Z$ can also be transformed into a decoupled generator by absorbing the linear coupling term into the Brownian motion using Girsanov's transformation. However, this treatment involves a change of probability measure \cite{Exarchos2018} and is left for future discussion.

In order to verify our theory, we test the four realizations of the propsed general criterion listed in Table~\ref{table:four-criteria} by the following example, which is modified based on Example~\ref{example:1}.

\begin{example}
  \label{example:2}
  Solve the BSDE~\eqref{eq:BSDE-couple-Z} with $t=0,T=1,f(\omega,s,z)=-1+\langle b_0X_s,Z_s\rangle, \xi=\langle X_T,X_T\rangle/n$, where $n=100$ is the dimension of the process $X$ and Brownian motion $W$. The process $X$ satisfies the stochastic differential equation: $X_s=W_s - \int_t^sb_0X_s\,ds$ with $b_0=-0.1$.
\end{example}

We parameterize the trial processes as $\tilde{v}_s=X_s^\intercal\theta_yX_s,~z_s = 2\theta_zX_s$. Other treatments remain unchanged from Example~\ref{example:1}. The true values can be verified by It\^o's formula as well: $\theta^*_y=\theta^*_z=1/n$. Results are reported in Figure~\ref{fig:2}.

\begin{figure}
  \centering
  \includegraphics[width=.4\textwidth]{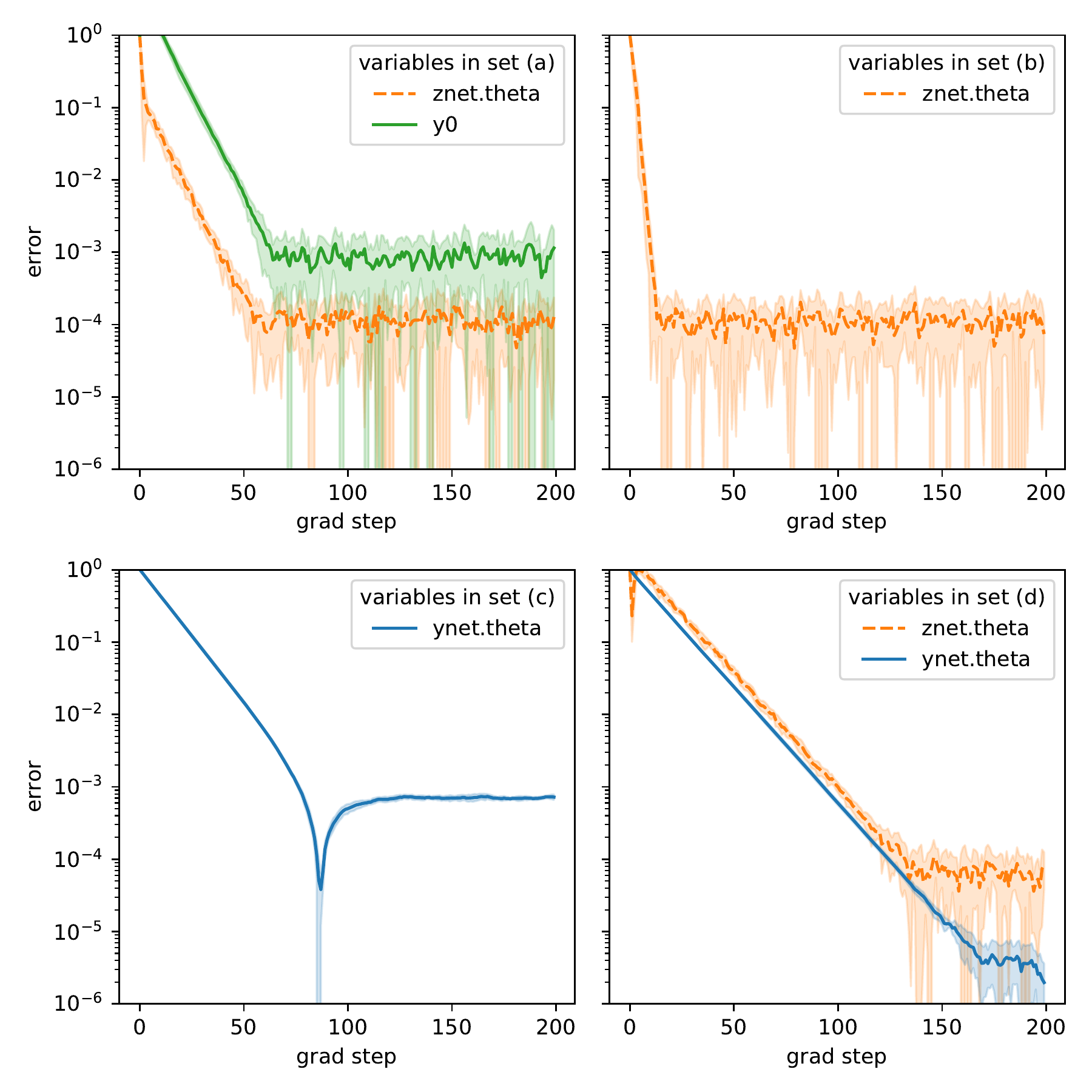}
  \caption{\label{fig:2}The absolute errors of $\theta_y,\theta_z,y_0$ at each gradient steps for Example~\ref{example:2}. From left to right and from top to bottom, the subplots correspond to Set (a), (b), (c) and (d). The solid lines and shaded areas indicate the mean and standard deviation of absolute errors for 10 runs.}
\end{figure}

\section{Simulation Results}\label{sec:simulations}

In this section, we test our on-policy and off-policy subroutines on a 100 dimensional optimal control problem. We obtain the $z$ function in these subroutines via optimizing the general BML criterion discussed in the last section. Specifically, we consider the four cases listed in Table~\ref{table:four-criteria}.

\begin{example}
  \label{example:3}
  Consider the following stochastic optimal control problem, which is an extension of the example in \cite{ji_peng_2020}:
  \begin{align*}
    \operatorname{minimize}&\quad \operatorname{\mathbb{E}}\biggl[\log \frac{1+\|X_T\|^2}{2} + \int_t^T\|\alpha_s\|^2\,ds \biggr],\\
    \operatorname{subject\ to}&\quad X_s = x + \int_t^s\sigma_0(\hat{b}_0\alpha_\tau\,d\tau + dW_\tau),\quad s\in[t,T],
  \end{align*}
  where $W$ is a standard 100 dimensional Brownian motion with $W_t=0$, and $\sigma_0,\hat{b}_0\in\mathbb{R}$ are positive constants. Determine the optimal cost when $x=0, t=0, T=1, \hat{b}_0=1$ and $\sigma_0=\sqrt{2}$.
\end{example}

We run the GPI equipped with Algorithm~\ref{alg:on-policy-FBSDE} and Algorithm~\ref{alg:off-policy-BSDE}. The initial policy is chosen to be $\alpha^0(t,x) = -0.1x$ and the behavior policy $\alpha^b$ is fixed to $\alpha^0$. In order to satisfy Assumption~\ref{assumption:1}.2, we manually clip the components of control inputs to $[-a_{\mathrm{max}},a_{\mathrm{max}}]$, which corresponds to the control set $A=[-a_{\mathrm{max}},a_{\mathrm{max}}]^{100}\subset\mathbb{R}^{100}$. In our experiments, $a_{\mathrm{max}}$ is set to 100. To simulate the forward process, we adopt the Euler-Maruyama method with time step size $\Delta t=0.01$\cite{Higham2001}. We optimize the proposed criterion on PyTorch platform\cite{PyTorch}. Table~\ref{table:four-criteria} is implemented with $\tilde{v}_s=\tilde{v}(s,X_s;\theta_y)$ and $z_s=z(s,X_s;\theta_z)$, where functions $\tilde{v}$ and $z$ are feed-forward neural networks with a single hidden layer. The width of hidden layers is set to 16. We use the SGD optimizer with Nesterov acceleration technique and momentum $1\times10^{-3}$\cite{Sutskever2013}. The optimization procedure is terminated after 75 gradient steps, and in each gradient step, the standard Euclidean norm of the total gradient is clipped to 10, and the learning rates are multiplied by a factor 0.99. Learning rates for $y_0, \theta_y, \theta_z$ are $0.5, 0.1, 0.1$, respectively. The sample size for estimating expectations is 16. For each criterion, we call the on-policy subroutine or the off-policy subroutine 9 consecutive times starting at $\alpha^0$. Results are reported in Figure~\ref{fig:3}.

\begin{figure}
  \centering
  \includegraphics[width=.4\textwidth]{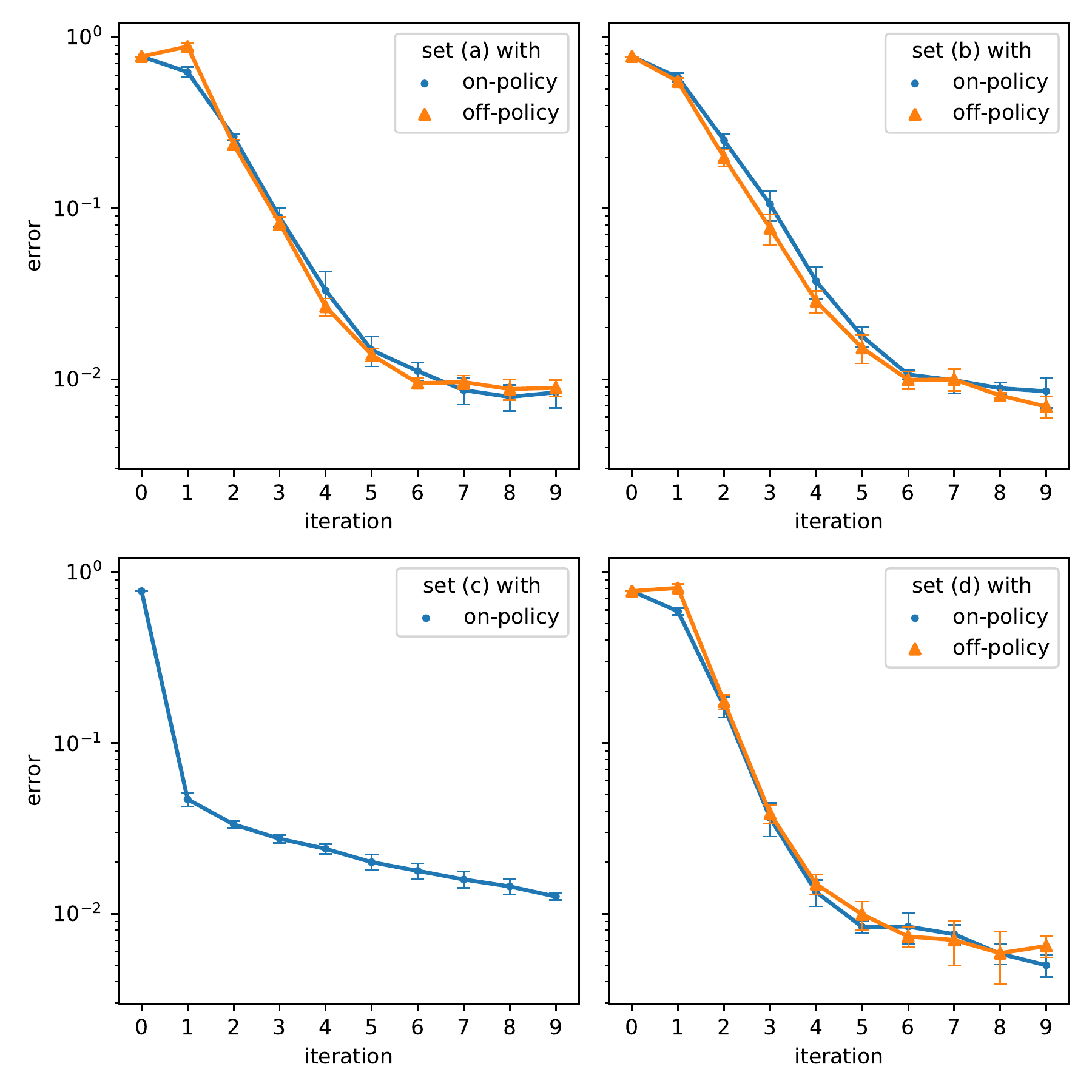}
  \caption{\label{fig:3}The absolute error between the optimal cost and $i$-th policy's cost for Example~\ref{example:3}. From left to right and from top to bottom, the subplots correspond to Set (a), (b), (c) and (d). Each subplot, except for Set (c), contains two lines representing the on-policy and the off-policy subroutines. The data points and error bars represent mean and standard deviation of 5 independent runs.}
\end{figure}

Figure~\ref{fig:3} plots the absolute error between the theoretical optimal cost and $i$-th policy's cost. The theoretical optimal cost is obtained by solving the associated HJB equation
\begin{equation*}
\left\{
\begin{aligned}
&\partial_tv^* + \frac{1}{2}\sigma_0^2\sum_{i=1}^{100}\partial_{x_ix_i}v^* - \frac{1}{4}\sigma_0^2\hat{b}_0^2\sum_{i=1}^{100}|\partial_{x_i}v^*|^2 = 0, \\
&v^*(T,x) = \log(1+\|x\|^2)/2.
\end{aligned}
\right.
\end{equation*}
Applying Hopf-Cole transformation to this equation yields the following representation of the solution\cite{Hopf1950}
\begin{equation*}
  v^*(t,x) = - \frac{2}{\hat{b}_0^2}\log \operatorname{\mathbb{E}}\biggl[\exp\Bigl( - \frac{\hat{b}_0^2}{2}\log \frac{1 + \|x+\sigma_0\epsilon\|^2}{2}\Bigr)\biggr],
\end{equation*}
where $\epsilon\in\mathbb{R}^{100}$ and is normally distributed with mean 0 and covariance matrix $(T-t)I$. We estimate this expectation by Monte Carlo with sample size $M=12800$. Figure~\ref{fig:3} shows that both the on-policy and off-policy subroutines and the four specific criteria can produce a good enough policy after 9 policy iteration steps. It is worth noting that there is no suitable off-policy method for the criterion of Set (c). This is due to the fact that the generator of the BSDE in Algorithm~\ref{alg:off-policy-BSDE} is explicitly coupled with $Z$, and thus, the optimization of $z$ and $\tilde{v}$ is not independent, cf. Remark~\ref{remark:decoupled-z-tildev}. Despite of this, we construct the improved policy by setting $z^\alpha=\sigma_0\partial_x\tilde{v}(\cdot,\cdot;\theta_y)$ in the on-policy subroutine for Set (c).

\begin{example}
  \label{example:4}
  Determine the optimal cost of Example~\ref{example:3} with $\sigma_0=20$.
\end{example}

Compared with the previous example, this only changes the system dynamics. Benefited from the data-driven nature of our algorithms, we can rerun the program with the only difference that trajectories are now sampled from this new system. Results are reported in Figure~\ref{fig:4}.

\begin{figure}
  \centering
  \includegraphics[width=.4\textwidth]{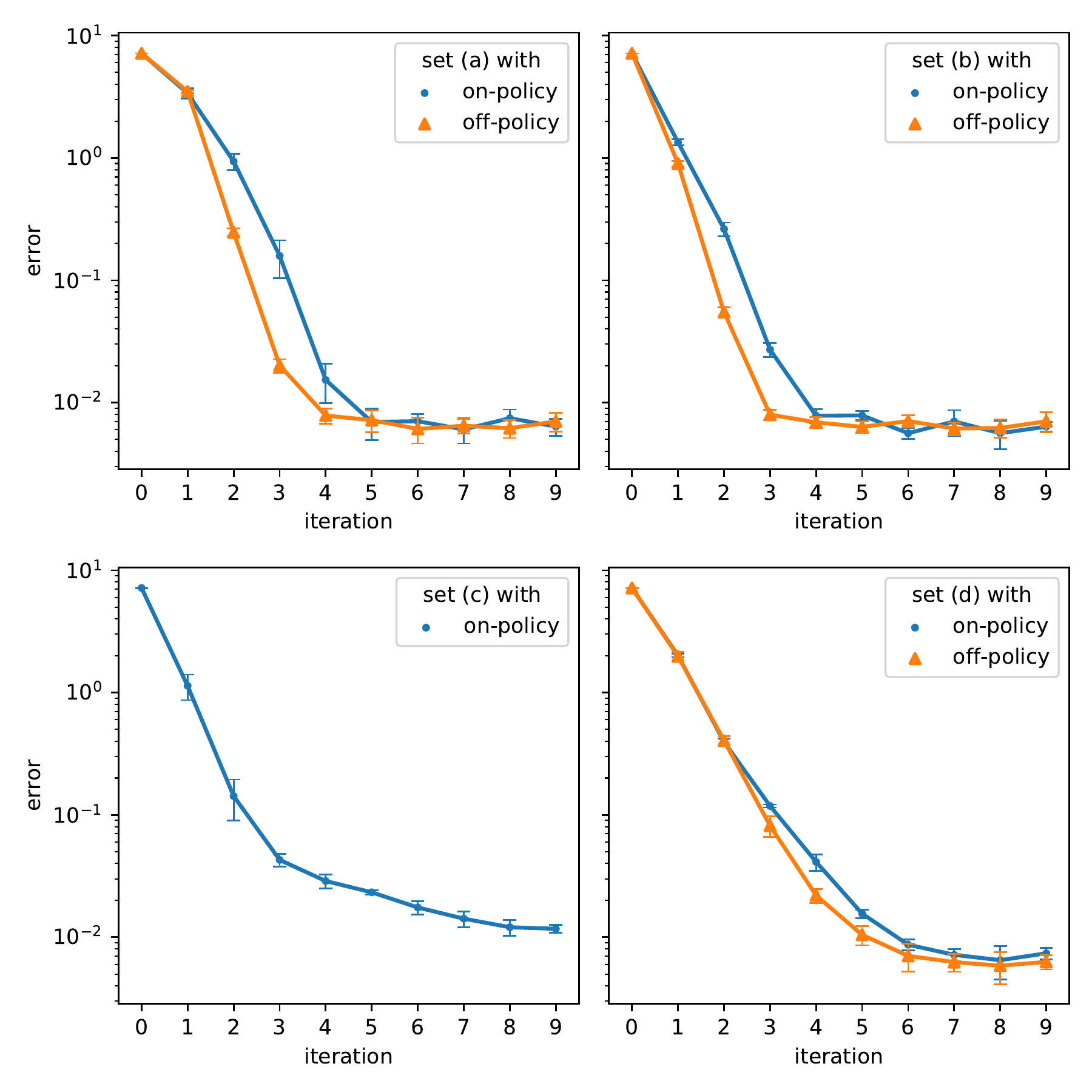}
  \caption{\label{fig:4}The absolute error between the optimal cost and $i$-th policy's cost for Example~\ref{example:4}. See Figure~\ref{fig:3} for the explanations of elements in figures.}
\end{figure}

\section{Conclusions and Outlook}\label{sec:conclusions}

In this paper, we build a probabilistic framework of Howard's policy iteration. Both an on-policy and an off-policy FBSDE-based GPI subroutines are proposed, and the convergence results associated with them are provided as well. In order to solve the proposed FBSDE-constrained optimization problem, we propose to minimize the BML criterion, which extends the Deep BSDE method in \cite{Han2018} and martingale approach \cite{jia_zhou_2022}. Future directions may include considering quadratic growth cost, relaxing the bounded control coupling term assumption and analyzing stability issues and so on.

\bibliographystyle{IEEEtran}
\bibliography{pibsde}

\end{document}